\documentclass[11pt,thmsa]{article}%
\usepackage{amsmath}
\usepackage{amsfonts}
\usepackage{amssymb}
\usepackage{graphicx}%
\newtheorem{theorem}{Theorem}[section]

\newtheorem{conjecture}[theorem]{Conjecture}
\newtheorem{corollary}[theorem]{Corollary}

\newtheorem{definition}[theorem]{Definition}

\newtheorem{lemma}[theorem]{Lemma}

\newtheorem{proposition}[theorem]{Proposition}

\newenvironment{proof}[1][Proof]{\noindent\textbf{#1.} }{\ \rule{0.5em}{0.5em}}

\begin{document}
\date{}

\title{$(\alpha_1,\alpha_2)$-Spaces and Clifford-Wolf Homogeneity\thanks{Supported by NSFC (no. 11271198, 11221091, 11271216) and SRFDP of China}}
\author{Ming Xu$^1$ and Shaoqiang Deng$^2$ \thanks{S. Deng is the corresponding author. E-mail: dengsq@nankai.edu.cn}\\
\\
$^1$College of Mathematics\\
Tianjin Normal University\\
 Tianjin 300387, P.R. China\\
 \\
$^2$School of Mathematical Sciences and LPMC\\
Nankai University\\
Tianjin 300071, P.R. China}
\maketitle

\begin{abstract}
In this paper, we introduce a new type of Finsler metrics, called $(\alpha_1,\alpha_2)$-metrics. We define the notion of  the good datum
of a homogeneous $(\alpha_1,\alpha_2)$-metric and use that to study the geometric properties.
In particular,  we give a formula of the S-curvature and deduce a condition for the S-curvature to be
vanishing identically. Moreover, we  consider the restrictive Clifford-Wolf homogeneity of
 left invariant  $(\alpha_1,\alpha_2)$-metrics on compact connected  simple Lie groups. We prove that, in some special cases, a restrictively Clifford-Wolf homogeneous $(\alpha_1,\alpha_2)$-metric must be Riemannian. An unexpected
interesting observation contained in the proof reveals the fact that the S-curvature may play an
important role in the study of Clifford-Wolf homogeneity in Finsler geometry.



\end{abstract}

\section{Introduction}

Finsler geometry has been proven to be very useful in many fields, including general relativity, medical imaging, psychology and biology. However, due to the complexity of the computation involved in the related problems, it is very hard to get deep results with the full generality. Meanwhile, the study of the special case, Riemannian geometry, is extremely fruitful. While studying non-Riemannian metrics, a large number of interesting results only deal with Randers metrics, a special class which is very close to Riemannian ones. See for example \cite{BCS00,CS04,DH13}, etc.

In view of the above facts, it would be much more promising that we first consider some special types of Finsler metrics. In this direction, at least two research fields in Finsler geometry are worthwhile to be mentioned, namely, the theory of $(\alpha,\beta)$-metrics and that of  homogeneous Finsler spaces.
An $(\alpha,\beta)$-metric is a Finsler metric of the form $F=\alpha\phi(\frac{\beta}{\alpha})$, where $\alpha$ is a Riemannian metric, $\beta$ is a $1$-form and $\phi$ is a real smooth function. This kind of metrics were first considered by Matsumoto. See \cite{MA92} for a survey of the development of the theory before 1992. The condition for such a metric to be positive definite can be found in \cite{CS04}. Recently, there is much progress in the study of $(\alpha,\beta)$-metrics; see for example \cite{SH09}.
The theory of homogeneous Finsler spaces has been developing very rapidly in the last decade. In \cite{DE12}, the second author gives a detailed survey on the topic and presents many open problems concerning Lie groups and Finsler geometry.

One main purpose of this paper is to initiate the study of a special class of Finsler metrics, called $(\alpha_1,\alpha_2)$-metrics. Roughly speaking,  an $(\alpha_1,\alpha_2)$-metric can be defined as follows. Let $\alpha$ be a Riemannian metric on a manifold $M$. Suppose  we have an $\alpha$-orthogonal decomposition
$TM=\mathcal{V}_1\oplus\mathcal{V}_2$, where $\mathcal{V}_1$ and $\mathcal{V}_2$ are
$n_1$- and $n_2$-dimensional linear subbundles of $TM$, respectively. Denote the restrictions of $\alpha$ to
$\mathcal{V}_1$ and $\mathcal{V}_2$ as $\alpha_1$ and $\alpha_2$, respectively. Then $\alpha_1$ and $\alpha_2$  can be naturally extended to $TM$ such that $\alpha^2=\alpha_1^2+\alpha_2^2$. A Finsler
metric $F=f(\alpha_1,\alpha_2)$, where $f$ is a positive smooth function on $\mathbb{R}^2$  which is positively homogenous of degree 1, is called an $(\alpha_1,\alpha_2)$-metric with dimension decomposition $(n_1,n_2)$. Note that we can always require that $n_1\geq n_2>1$, since otherwise
we can exchange the two subbundles if $n_1<n_2$, and it is an $(\alpha,\beta)$-metric when
$n_2=1$.

The notion of $(\alpha_1,\alpha_2)$-metrics can be viewed as a generalization of
$(\alpha,\beta)$-metrics.
Although we expect that $(\alpha_1,\alpha_2)$-metrics will be as computable as $(\alpha,\beta)$-metrics, it will be more convenient that
we start with  homogeneous $(\alpha_1,\alpha_2)$-spaces. We will define the good normalized datum of an $(\alpha_1,\alpha_2)$-metric and show how to find good normalized datum for the metric
which is compatible with its homogeneity. As an application, we obtain a formula of the S-curvature of a homogeneous $(\alpha_1,\alpha_2)$-space and deduce a condition for the S-curvature to be vanishing identically.

Another main purpose of this paper is to discuss the restrictive Clifford-Wolf homogeneity (restrictive CW-homogeneity) of left invariant $(\alpha_1,\alpha_2)$-metrics $F$ on a compact connected simple Lie group $G$. The motivation to study the CW-homogeneity
of left invariant Finsler metrics is as the following. Recall that  connected CW-homogeneous Riemannian manifolds have been classified by Berestovskii and Nikonorov in \cite{BN09}. The list consists of the Euclidean spaces, the odd-dimensional spheres of constant curvature, compact Lie groups with bi-invariant Riemannian metrics, and the Riemannian product of the above three types of manifolds. Therefore, to classify CW-homogeneous Finsler spaces, it is natural to consider CW-homogeneous Finsler metrics on the above three types of manifolds.
The analogue of Euclidean spaces in Finsler geometry is easy, and CW-homogeneous Finsler metrics on spheres have been classified recently
by the authors in \cite{XD8}. Therefore our next step is to classify left invariant CW-homogeneous Finsler metrics on compact Lie groups. In our previous works, we have  performed this program for Randers metrics and $(\alpha,\beta)$-metrics (see \cite{DM3, DM5}). Therefore we study CW-homogeneity of left invariant $(\alpha_1,\alpha_2)$-metrics on compact Lie groups in this paper. However, the general case seems to be very involved, so
we will confine ourselves to
the case that in the decomposition
$\mathfrak{g}=\mathbf{V}_1+\mathbf{V}_2$ of $TG_e=\mathfrak{g}$, the subspace $\mathbf{V}_2$ is a commutative subalgebra
of $\mathfrak{g}$. In particular,  we will discuss the following two cases:
\begin{description}
\item{{\bf Case 1}.}\quad $G$ is a compact connected simple Lie group, and $\mathbf{V}_2$ is a Cartan subalgebra.
\item{{\bf Case 2}.}\quad  $G$ is a compact connected simple Lie group, and $\mathbf{V}_2$ is 2-dimensional commutative subalgebra.
\end{description}

In the study of the restrictive CW-homogeneity of  left invariant non-Riemannian
$(\alpha_1,\alpha_2)$-metrics in Case 1, the $S$-curvature plays an important role.
The main results  are the following two theorems.

\begin{theorem}\label{main-thm-0}
Let $F$ be a left invariant  $(\alpha_1,\alpha_2)$-metric on a compact connected simple Lie group $G$ with a decomposition $\mathfrak{g}=\mathbf{V}_1+\mathbf{V}_2$ such that $\mathbf{V}_2$ is a Cartan subalgebra,  and $\dim \mathbf{V}_2>1$. If $F$ is  restrictively CW-homogeneous, then it must be Riemannian.
\end{theorem}

The study of the restrictive CW-homogeneity of left invariant non-Riemannian $(\alpha_1,\alpha_2)$-metrics of Case 2 is a generalization of
our work on $(\alpha,\beta)$-metrics \cite{DM5}.
The systematic technique we have developed in the study of
Killing vector fields of constant length  of left invariant Randers and $(\alpha,\beta)$-metrics also works in
this case.

\begin{theorem}\label{main}
Let $F$ be a left invariant  $(\alpha_1,\alpha_2)$-metric on a compact connected simple Lie group $G$ with a decomposition $\mathfrak{g}=\mathbf{V}_1+\mathbf{V}_2$ such that $\mathbf{V}_2$ is a $2$-dimensional commutative subalgebra of dimension $\geq 2$. If  $F$  is restrictively CW-homogeneous, then it must be Riemannian.
\end{theorem}

Theorem \ref{main-thm-0} and Theorem \ref{main} suggest us make the following conjecture:
\begin{conjecture}
Let $F$ be a left invariant  $(\alpha_1,\alpha_2)$-metric on a compact connected simple Lie group $G$ with a decomposition $\mathfrak{g}=\mathbf{V}_1+\mathbf{V}_2$ such that $\mathbf{V}_2$ is a commutative subalgebra. If $F$ is restrictively CW-homogeneous, then it must be Riemannian.
\end{conjecture}
More generally, we can make a stronger one:
\begin{conjecture}
Let $F$ be a left invariant  $(\alpha_1,\alpha_2)$-metric on a compact connected simple Lie group $G$ with a dimension decomposition $(n_1,n_2)$, where $n_1\geq n_2>1$. If $F$ is restrictively CW-homogeneous, then it must be Riemannian.
\end{conjecture}


The paper is organized as follows.
In Section 2, we recall some definitions and known results in Finsler geometry. Section 3 is devoted to defining $(\alpha_1,\alpha_2)$-metrics and introducing the good normalized datum
of a non-Riemannian homogeneous $(\alpha_1,\alpha_2)$-space. In Section 4, we present a S-curvature
formula of homogeneous $(\alpha_1,\alpha_2)$-spaces. In Section 5, we use the S-curvature
to  study the restrictive CW-homogeneity of left invariant non-Riemannian
$(\alpha_1,\alpha_2)$-metrics of Case 1 and prove  Theorem \ref{main-thm-0}. In Section 6, we study restrictive CW-homogeneity of left invariant
$(\alpha_1,\alpha_2)$-metrics of Case 2 and prove  Theorem \ref{main}.
Finally, in Section 7, we prove a key lemma  stated in Section 6, completing the proof of all the main results in this paper.

\section{Preliminaries}
\subsection{Minkowski norms and Finsler metrics}
In this section, we recall the notions of Minkowski norms and Finsler metrics.
In this paper, all manifolds are assumed to be connected and smooth.

\begin{definition}
A Minkowski norm on a $n$-dimensional real vector space $\mathbf{V}$ is a real continuous function $F$ on
$\mathbf{V}$  which is smooth on $\mathbf{V}\backslash\{0\}$ and satisfies the following
conditions:
\begin{description}
\item{\rm 1.}\quad $F$ is non-negative: $F(u)\geq 0$, $\forall u\in \mathbf{V}$;
\item{\rm 2.}\quad  $F$ is positively homogeneous of degree one: $F(\lambda u)=\lambda F(u)$, $\forall\lambda >0$;
\item{\rm 3.}\quad $F$ is strongly convex. Namely, given a basis $u_1, u_2, \cdots, u_n$ of
$\mathbf{V}$, write
$F(y)=F(y^1, y^2, \cdots, y^n)$ for
$y=y^1u_1+y^2u_2+\cdots+y^nu_n$. Then the Hessian matrix
$$(g_{ij}):=\left([\frac{1}{2}F^2]_{y^iy^j}\right)$$
is positive-definite at any point of $V\backslash\{0\}$.
\end{description}
\end{definition}

\begin{definition}
 A Finsler metric on a manifold $M$ is a continuous function $F$: $TM\to
[0,\infty)$ such that
\begin{description}
\item{\rm 1)}\quad $F$ is smooth  on the slit tangent bundle $TM\backslash 0$;
\item{\rm 2)}\quad The restriction of $F$ to any tangent space  $T_xM$, $x\in M$,  is a
Minkowski norm.
\end{description}
In this case we say that $(M, F)$ is a Finsler space.
\end{definition}

For fundamental properties of Finsler spaces, we refer to \cite{BCS00, CS04} and \cite{SH2}.

  On a Finsler space $(M, F)$ the distance function $d(\cdot,\cdot)$ can be  defined similarly as in the Riemannian case. Note that  the reversibility of $d(\cdot,\cdot)$, i.e., the condition that  $d(x,x')=d(x',x)$, for  $x,x'\in M$,  may not be satisfied.

Here are some examples which are relevant to this work.

A Randers metric $F$ is defined as $F=\alpha+\beta$, where $\alpha$ is a Riemannian metric, and $\beta$ is an 1-form with  $\alpha$-length  less than 1 everywhere. Randers metrics are computable and it has always been a central focus in  Finsler geometry. Moreover, Randers metrics have been applied to many scientific fields, see for example \cite{BCS00}.

An $(\alpha,\beta)$-metric $F$ is defined as $F=\alpha\phi(\frac{\beta}{\alpha})$, where $\phi$ is a real function on $\mathbb{R}$,  $\alpha$ is a Riemannian metric and $\beta$  is a 1-form on $M$. Note that for any $x\in M$ and $Y\in T_x(M)\backslash\{0\}$, $\frac{\beta (Y)}{\alpha (Y)}\in \mathbb{R}$, and it is positive homogeneous of degree $0$. Therefore $F$ is positive homogeneous of degree $1$. For the smoothness and strong convexity of $F$,  the function $\phi$ is required to be smooth anywhere involved in the definition of the metric, and  satisfy the following inequality:
$$\phi(s)-s\phi'(s)+(b^2-s^2)\phi''(s)>0,$$
for any $s$ and $b$ such that there is an $x\in M$ with $|s|\leq |b|\leq ||\beta(x)||_{\alpha(x)}$ (see \cite{CS04}). Obviously, a Randers metric is an $(\alpha,\beta)$-metric with $\phi(s)=1+s$ Note that  $\phi$ can also be of the form $\sqrt{k_1+k_2s^2}+k_3s$, where $k_1, k_2, k_3$ are constants.

\subsection{Curvatures in Finsler geometry}
Curvature is the most important concept in geometry. In Finsler geometry, there are a lot
of curvatures, including Riemannian ones and non-Riemannian ones. Here we
only mention those curvatures related to the topics in this paper.

First we recall the concept of some the non-Riemannian curvatures.

 Let $(M,F)$ be an $n$-dimensional Finsler space. The Busemann-Hausdorff
volume can be globally defined on $M$ as follows.
Given a local coordinates $\{x=(x^i)\in M,y=y^j\partial_{x^j}\in TM_x\}$, we define $dV_{BH}=\sigma(x)dx^1\cdots dx^n$, where
$$\sigma(x)=\frac{\mbox{Vol}(B_n(1))}{\mbox{Vol}\{(y^i)\in\mathbb{R}^n|F(x,y^i\partial_{x^i})<1\}},$$
here $\mbox{Vol}$ is the volume with respect to the standard Euclidian metric on
$\mathbb{R}^n$, and $B_n(1)$ is the  unit ball in $\mathbb{R}^n$.
Although the coefficient function $\sigma(x)$ is only locally defined and
depends on the choice of local coordinates $x=(x^i)$, the distortion function
\begin{equation}\label{-6}
\tau(x,y)=\ln\frac{\sqrt{\det(g_{ij}(x,y))}}{\sigma(x)}
\end{equation}
on $TM\backslash 0$ is independent of the local coordinates and is globally defined.

The S-curvature $S(x,y)$ is a function on $TM\backslash 0$ which is defined to be the derivative of
$\tau(x,y)$ in the direction of the geodesic spray, which is also a globally defined
vector field on $TM\backslash 0$. In local coordinates, the geodesic spray can be represented as
$G=y^i\partial_{x^i}-2G^i\partial_{y^i}$, where
$$G^i=\frac{1}{4}g^{il}({[F^2]}_{x^k y^l}y^k-{[F^2]}_{x^k}).$$

The derivatives of $\tau(x,y)$ in the $y$-direction define another non-Riemannian
curvature called the mean Cartan tensor. In local coordinates, it can
be represented as
$$I_y(u)=u^i\partial_{y^i}\ln\sqrt{\det(g_{pq}(y))}, \forall u=u^i\partial_{y^i}.$$
Recall that $F$ is a Riemannian metric if and only if the mean Cartan
tensor vanishes identically \cite{D53}.

Now we turn to the Riemannian curvature.

The Riemannian curvature can be defined using either the Jacobi field or  the structure
equations. In local coordinates, it can be interpreted as the linear transformations of $T_x(M)$ defined by
$R_y=R^i_k(y)\partial_{x^i}\otimes dx^k:T_x M\rightarrow T_x M$
 for any nonzero tangent vector $y\in T_x M$, where
$$R^i_k(y)=2\partial_{x^k}G^i-y^j\partial^2_{x^j y^k}G^i+2G^j\partial^2_{y^j y^k}G^i
-\partial_{y^j}G^i\partial_{y^k}G^j.$$

Furthermore, let $P$ be a tangent plane in $T_x M$ containing a nonzero vector $y$. Suppose $P$ is spanned   by $y$ and $u$. Then the flag curvature of the flag $(P, y)$ is  defined as
$$K(P,y)=\frac{\langle R_y(u),u\rangle_y}
{\langle y,y\rangle_y\langle u,u\rangle_y-\langle y,u\rangle_y^2},$$
where $\langle \cdot,\cdot\rangle_y$ is the inner product defined by $(g_{ij}(y))$.
Obviously the flag curvature is the generalization of Riemannian sectional curvature
in Finsler geometry.

Shen indicated the following important observation which gives an elegant expression of the Riemannian curvature of
a Finsler space in terms of  the Riemannian curvature of the osculating Riemannian metrics \cite{SH3}.

\begin{proposition}\label{proposition-1}
Let $F$ be a Finsler metric on $M$, and $Y$ be a non-zero geodesic field defined on an open subset $U$ containing $x\in M$. Denote the Riemannian metric defined by $(g_{ij}(Y(\cdot)))$
as $\hat{g}$.
Then we have $R_y=\hat{R}_y$, where $y=Y(x)$, and $\hat{R}_y$ is the Riemannian curvature of $\hat{g}$.
\end{proposition}

In particular, let $P\subset T_x M$ be a tangent plane containing $y$. Denote the sectional curvature of $\hat{g}$ as $\hat{K}$. Then  it follows from Proposition \ref{proposition-1}  that $K(P,y)=\hat{K}(P)$.

\subsection{Homogeneous Fisnler space}

An isometry $\rho$ of a Finsler metric $F$ on a manifold $M$ is a diffeomorphism of $M$ which preserves $F$, i.e., $\rho^* F=F$. This is equivalent to the condition that $\rho$ preserves the distance function $d(\cdot,\cdot)$, i.e., $d(\rho(x),\rho(x'))=d(x,x')$ for
any $x,x'\in M$; see \cite{DH02}.

The group of all isometries of $(M,F)$ is denoted as $I(M,F)$. In \cite{DH02} it is  proved  that $I(M, F)$  is a Lie transformation group of $M$. Now we give a different  point of view of this important result, namely, $I(M,F)$ can be viewed as a closed subgroup of the isometry group of a Riemannian metric on   $M$. More precisely,  define a Riemannian metric $F'$ on $M$ by averaging $F$ as the following:
$$F'^2(x,y')=\int_{F(y)=1}\langle y',y'\rangle_y d\mbox{vol}_y,$$
where $d\mbox{vol}$ is the volume form of the indicatrix in $T_x M$, endowed with
the induced Riemannian metric defined by the Hessian matrix $(g_{ij})$ on $T_x M\backslash \{0\}$.
Then any isometry of $(M,F)$ must be an isometry of $(M,F')$. Moreover, it is easily seen that $I(M,F)$ is a closed
subgroup of $I(M,F')$. Therefore $I(M, F)$ is  a Lie transformation group of $M$; see \cite{Hel}.

The maximal connected subgroup of $I(M, F)$  is called the connected isometry group of $(M,F)$, and  is denoted as $I_0(M,F)$. The Finsler space $(M,F)$ is called  homogeneous  if $I(M,F)$ acts transitively on $M$. Since $M$ is connected, the above definition is equivalent to the condition that $I_0(M, F)$ is transitive on $M$ \cite{Hel}. In general, there may exist  some proper subgroups of $I_0(M, F)$ which are also transitive on $M$. For any closed subgroup $G$ of $I_0(M,F)$ which acts transitively on $M$, the manifold $M$ can  be written as the quotient $M=G/H$, where $H$ is the isotropy subgroup of $G$ at a fixed point  $x\in M$. Denote $\mathfrak{g}=\mathrm{Lie}(G)$, $\mathfrak{h}=\mathrm{Lie}(H)$. Then the tangent space  $T_x M$ can be identified with the quotient space $\mathfrak{m}=\mathfrak{g}/\mathfrak{h}$.

As an explicit example, let $F$ be a left invariant metric $F$ on a connected Lie group $G$. Then $(G, F)$  is a homogeneous Finsler space,
since the  connected isometry group $I_0(G,F)$ contains $L(G)$, which acts transitively
on $G$. Note that in general the full connected group of isometries of $(G, F)$ is larger than $L(G)$; see \cite{OT76}.

The study of homogeneous Finsler space helps us  understand the intrinsic nature of Finsler geometry without carrying out  complicated calculations on curvatures and tensors, since the related quantities can  generally  be reduced to  tensor vectors on the tangent space $\mathfrak{m}$. For example, a homogeneous Randers metric $F=\alpha+\beta$ on $M=G/H$ can be determined by the restrictions of $\alpha$ and $\beta$ on $\mathfrak{m}$, which is
an $\mathrm{Ad}(H)$-invariant linear metric and an $\mathrm{Ad}(H)$-invariant vector in $\mathfrak{m}^*$, respectively. The isometry group $I(M,F)$ consists of the elements in  $I(M,\alpha)$ which preserve the $1$-form  $\beta$. However,  when dealing with $(\alpha,\beta)$-metrics, the situation becomes much more complicated. This is mainly due to the fact that the representation of the metric $F$ by $\alpha$, $\beta$ and $\phi$ is generally
not unique. In fact, it may happen that, when we write a $G$-invariant $(\alpha,\beta)$-metric $F$ on a coset space $G/H$ as $F=\alpha\phi(\frac{\beta}{\alpha})$, neither the globally defined $\alpha$  nor the $1$-form $\beta$ is
$G$-invariant. This means that  none of  their restrictions in $\mathfrak{m}$ is $\mathrm{Ad}(H)$-invariant. We will meet a similar situation in the study of  $(\alpha_1,\alpha_2)$-metrics in this
paper. To settle this problem, we introduce the notion of a good datum for a homogeneous Finsler
metric.
%
%
%

Let $F$ be a homogeneous Finsler metric on a manifold $M$. In many cases, we need to use  some  data to define the  metric $F$. For example, if $F$ is a Randers metric, then we need  a  pair $(\alpha,\beta)$; if $F$ is an $(\alpha,\beta)$-metric, then we need a triple $(\phi,\alpha,\beta)$. The datum
is called a good datum if it is invariant under the action of $I_0(M,F)$.
It is obvious that the restriction of a good datum to the tangent space $T_x(M)$ is invariant under the action of the isotropy subgroup $I_x(M, F)$.

For convenience, we will usually  use the same notations to denote a good datum and its restriction to $T_xM$.
The reason to introduce the notion of a good datum is that, using a good datum for the homogeneous Finsler metric, we can safely reduce the computation to the given tangent space without losing controls of its global homogeneity. It is obvious that  for a homogeneous Randers metric $F=\alpha+\beta$,
the pair $(\alpha,\beta)$ is always a good datum.


\subsection{CW-translations and CW-homogeneity of Finsler spaces}

 We first recall the definition of a Clifford-Wolf translation  on a Finsler space. An isometry $\rho$ of $(M,F)$ is called a Clifford-Wolf translation (CW-translation for short) if it moves all points  the same distance, i.e., $d(x,\rho(x))= \mbox{const}$.
Similarly as in the Riemannian case,  we can consider the Clifford-Wolf homogeneity in Finsler geometry.

\begin{definition}
A Finsler space $(M,F)$ is called Clifford-Wolf homogeneous (CW-homogeneous) if for any two points $x,x'\in M$, there is a CW-translation $\rho$ such that $\rho(x)=x'$.
\end{definition}

A CW-homogeneous Finsler space is obviously a homogeneous Finsler space.
The study of CW-translations and CW-homogeneity in Finsler geometry was initiated in our
consideration (see \cite{DM1}-\cite{DM3} and \cite{DM4}) on the interrelation between CW-translations and Killing vector fields of
constant length (KVFCLs for short),  generalizing the work of V.N. Berestovskii and
Yu.G.Nikonorov in the Riemannian case; see \cite{BN081,BN082,BN09}. We refer to \cite{Wo1, WO64, Wo2} for more information on the study of CW-translations on Riemannian manifolds.
 Since in this paper we will consider  only compact Finsler spaces, we just restate Theorem 3.3 and Theorem 3.4 in \cite{DM1} as the following theorem.

\begin{theorem}\label{theorem-1}
 Let $(M,F)$ be a compact Finsler space and $X$ be a KVFCL on $M$. Then the one-parameter group of transformations $\varphi_t$ generated by $X$  are CW-translations when $t>0$ is sufficiently small. Moreover, there exists a  neighborhood $\mathcal{N}$ of the identity map in $I_0(M,F)$, such that any CW-translation $\rho\in\mathcal{N}$ is generated by a KVFCL with  $\varphi_t\in \mathcal{N}$ for
$t\in [0,1]$ and $\varphi_1=\rho$. Furthermore, all $\varphi_t$s are CW-translations
for $t\in [0,1]$.
\end{theorem}

There is a weaker version of  the CW-homogeneity, called restrictive CW-homogeneity. In this paper, we will mainly   deal with restrictive CW-homogeneity.

\begin{definition}
A compact Finsler space $(M,F)$ is called restrictively CW-homogeneous if there exists $\delta >0$ such that for any $x,x'\in M$, with $d(x,x')<\delta$, there exists  a CW-translation $\rho$ such that $\rho (x)=x'$.
\end{definition}

Theorem \ref{theorem-1} provides
 an explicit and equivalent description of restrictive CW-homogeneity for compact Finsler spaces. More precisely, we have

\begin{proposition}\label{proposition-0}
A compact Finsler space $(M,F)$ is restrictive CW-homogeneous if and only if  any tangent vector of $M$ can be extended to a KVFCL.
\end{proposition}

\section{Defining  $(\alpha_1,\alpha_2)$-metrics}

\subsection{The local model:  $(\alpha_1,\alpha_2)$-norms}

Before defining  global $(\alpha_1,\alpha_2)$-metrics, let us look at the local model. In the following, by a  Euclidean norm  on the linear space $\mathbb{R}^n$ we mean a metric of the form $\alpha (X)=\sqrt{\langle X,X\rangle}$, where $\langle,\rangle$ is an inner product on $\mathbb{R}^n$.
\begin{definition}
\label{def-alpha-1-2}
A Minkowski norm $F$ on $\mathbb{R}^n$, $n>3$, is called an $(\alpha_1,\alpha_2)$-norm with  dimension decomposition $(n_1,n_2)$, where $n_1\geq n_2>1$, if we can
find a Euclidean norm$\alpha$ on $\mathbb{R}^n$, an $\alpha$-orthogonal decomposition $\mathbb{R}^n=\mathbf{V}_1\oplus \mathbf{V}_2$, where $\dim \mathbf{V}_1=n_1$,  $\dim \mathbf{V}_2=n_2$, such that the value of $F(y)$ is a function  of $\alpha(y_1)$ and $\alpha(y_2)$, where $y=y_1+y_2$ is the decomposition of $y$ with respect to the above decomposition of $\mathbb{R}^n$.
\end{definition}

We exclude the cases that $n_2=0$ and $n_2=1$ since in those cases $F$ is either a Euclidean norm or  an $(\alpha,\beta)$-norm.

The restriction of $\alpha$ to $\mathbf{V}_1$ and $\mathbf{V}_2$ are denoted as $\alpha_1$ and $\alpha_2$, respectively. Moreover, the $\alpha$-orthogonal projections to
$\mathbf{V}_1$ and $\mathbf{V}_2$ are denoted as $\mathrm{pr}_1$ and
$\mathrm{pr}_2$ respectively.
Composed with $\mathrm{pr}_1$ and $\mathrm{pr}_2$,
$\alpha_1$ and $\alpha_2$ can be regarded as the square roots of  positive semi-definite quadratic functions on $\mathbb{R}^n$. Then we have
$\alpha^2=\alpha_1^2+\alpha_2^2$.

An $(\alpha_1,\alpha_2)$-norm can also be represented as $F=f(\alpha_1,\alpha_2)$. By the homogeneity, we have
$$
F=\alpha f(\frac{\alpha_1}{\alpha},\frac{\alpha_2}{\alpha})=\alpha f(\sqrt{1-(\frac{\alpha_2}{\alpha})^2},\frac{\alpha_2}{\alpha}).
$$
Thus
we can also denote it as $F=\alpha\phi(\frac{\alpha_2}{\alpha})$. Similarly,   we can also write
$F=\alpha\psi(\frac{\alpha_1}{\alpha})$. Since only the values of $\phi$ and $\psi$ on $[0,1]$ are relevant to $F$ and they must be positive, we can  assume that they are positive functions on $[0,1]$. It is easily seen that
$\phi(s)=\psi(\sqrt{1-s^2})$.
The following theorem gives more  requirements on the functions $\phi$ and $\psi$.

\begin{theorem} \label{define-requirement}
 Keeping all the notations as in Definition \ref{def-alpha-1-2}, we have
\begin{description}
\item{\rm (1)}\quad    Let $\phi$ and $\psi$ be two positive functions on $[0,1]$ such that
$\phi(s)=\psi(\sqrt{1-s^2})$. Suppose   $F=\alpha\phi(\frac{\alpha_2}{\alpha})=\alpha\psi(\alpha_1/\alpha)$ defines an $(\alpha_1,\alpha_2)$-norm on $\mathbb{R}^n$
with dimension decomposition $(n_1,n_2)$, where $n_1\geq n_2>1$.
Then
both $\phi$ and $\psi$ are positive smooth functions on $[0,1]$, and
\begin{eqnarray}
\phi(s)-s\phi'(s)+(b^2-s^2)\phi''(s) >0, \label{-1} \\
\psi(s)-s\psi'(s)+(b^2-s^2)\psi''(s)>0, \label{-2}
\end{eqnarray}
for any $s$ and $b$ with $0\leq s\leq b\leq 1$.
\item{\rm (2)}\quad Conversely, let $\phi$ and $\psi$ be two positive smooth functions on $[0,1]$
such that $\phi(s)=\psi(\sqrt{1-s^2})$. If
\begin{equation}
\phi(s)-s\phi'(s)+(1-s^2)\phi''(s)>0,
\end{equation}
then $F=\alpha\phi(\frac{\alpha_2}{\alpha})=\alpha\psi(\frac{\alpha_1}{\alpha})$
defines an $(\alpha_1,\alpha_2)$-norm with dimension decomposition
$(n_1,n_2)$, where $n_1\geq n_2>1$.
\end{description}
\end{theorem}

\begin{proof}
(1)\quad  Assume that $F=\alpha\phi(\frac{\alpha_2}{\alpha})=\alpha\psi(\frac{\alpha_1}{\alpha})$
defines an $(\alpha_1,\alpha_2)$-norm on $\mathbb{R}^n$, with  dimension decomposition $(n_1, n_2)$, where
$n_1\geq n_2>1$. Fix an orthonormal basis of $\mathbb{R}^n$ with respect to $\alpha$ such that the first $n_1$ vectors are from $\mathbf{V}_1$ and the others are from $\mathbf{V}_2$. Let $(y^1, y^2,\ldots, y^n)$ be  the
corresponding linear coordinates  and consider the circle $y(t)=(\cos t,0,\ldots,0,\sin t)$. Then the restriction of  $F=\alpha\phi(\frac{\alpha_2}{\alpha})$ to this circle is
\begin{equation}
F(y(t))=\phi(|\sin t|)=\phi(\frac{\alpha_2(y(t))}{\alpha(y(t))}).
\end{equation}
For $t\in (-\pi/2,\pi/2)$, where $t=\arcsin s$ is a smooth function of $s=\sin t$, the
even extension of $\phi$ must be a positive and smooth function on $(-1,1)$. Repeatedly
using L'Hospital rule, one easily sees that $\phi(s)$ is a smooth function of $\tilde{s}=s^2$ for
$\tilde{s}\in [0,1)$. Similarly, $\psi(s)$ is a smooth function of $\tilde{s}=s^2$ for
$\tilde{s}\in [0,1)$. By the relation
$\phi(s)=\psi(\sqrt{1-s^2})$,  $\phi(s)$ is smooth at $s=1$. Similarly,   $\psi(s)$ is   smooth  at $s=1$.
Therefore  $\phi$ and $\psi$ are positive and smooth functions on $[0,1]$.

The strong convexity
 of Finsler metrics is equivalent to the condition that,
when the indicatrix is viewed as a
hypersurface in $\mathbb{R}^n$ with the metric induced from flat metric $\alpha$, if we fix the outside
unit normal field, then the principal curvatures are all positive. To calculate the principle curvatures of the
indicatrix, we parameterize the indicatrix as $(\sqrt{1-s^2}u\phi^{-1}(s),sv\phi^{-1}(s))$,
where $s\in [0,1]$, and $u, v$ are the parameters on the $(n_1-1)$- and $(n_2-1)$-dimensional unit spheres in $\mathcal{V}_1$ and $\mathcal{V}_2$, respectively. For $s\in (0,1)$, it provides good coordinates on the indicatrix.  Now
the principal curvature of the $s$-curve is
\begin{equation}\label{-3}
\frac{\phi(s)-s\phi'(s)+(1-s^2)\phi''(s)}{(\frac{(1-s^2){\phi'}^2(s)}{\phi^2(s)}+1)^{3/2}}.
\end{equation}
In the directions with $u$ changing and $s, v$ fixed, we get $n_1-1$ principle curvatures
\begin{equation}\label{-4}
\frac{\phi(s)-s\phi'(s)}{((1-s^2)\frac{\phi'^2(s)}{\phi^2(s)}+1)^{1/2}}.
\end{equation}
In the directions with $v$ changing and $s, u$ fixed, we get $n_2-1$ principle curvatures
which have a similar expression as (\ref{-4}), namely, we just need to replace $\phi$ with $\psi$, and replace $s$ with $\bar{s}=\sqrt{1-s^2}$:
\begin{equation}\label{-5}
\frac{\psi(\bar{s})-\bar{s}\psi'(\bar{s})}{((1-\bar{s}^2)
\frac{\psi'^2(\bar{s})}{\psi^2(\bar{s})}+1)^{1/2}}.
\end{equation}
By the continuity of the principle curvatures, (\ref{-3})-(\ref{-5}) also give all principal curvatures when $s=0$ or 1.
It is not hard to see that
\begin{equation}\label{f-3-17}
\phi(s)-s\phi'(s)+(1-s^2)\phi''(s)=\psi(\bar{s})-\bar{s}\psi'(\bar{s})+(1-\bar{s}^2)\psi''(\bar{s}),
\end{equation}
so the condition given by (\ref{-1}) and (\ref{-2}) for all $s$ and $b$ with $0\leq s\leq b\leq 1$ is equivalent to the positiveness of
(\ref{-3})-(\ref{-5}) as well as to the positiveness of all principal curvatures.

(2)\quad  If  $\phi$ and $\psi$ are positive smooth functions related by
$\phi(s)=\psi(\sqrt{1-s^2})$, then $\{\frac{1}{\phi(s)}(x,y)| |y|=s\in [0,1], |x|=\sqrt{1-s^2}\}$ defines a closed smooth curve in $\mathbb{R}^2$. Since
$\phi(s)-s\phi'(s)+(1-s^2)\phi''(s)>0$, $\forall s\in [0,1]$, its curvature is nonzero everywhere, or equivalently, it is the boundary of a strictly convex region. Then we  have $\phi(s)-s\phi'(s)>0$.
Thus $\phi(s)-s\phi'(s)+(b^2-s^2)\phi''(s)>0$, for $0\leq |s|\leq |b|\leq 1$, that is,  (\ref{-1})
is satisfied.
On the other hand, by (\ref{f-3-17}) we also have $\psi({s})-{s}\psi'({s})+(1-{s}^2)\psi''({s})>0$,
$\forall s\in [0,1]$, hence (\ref{-2}) is also satisfied.

Using a similar
argument as in (1) one easily shows that $F=\alpha\phi(\frac{\alpha_2}{\alpha})=\alpha\psi(\frac{\alpha_1}{\alpha})$ is positive and smooth on
$TM\backslash 0$. Now the first two conditions of Minkowski norms are obviously satisfied, and the last condition
is guaranteed by the inequalities (\ref{-1}) and (\ref{-2}). This completes the proof of the theorem.
\end{proof}

Sometimes we need to write an $(\alpha_1,\alpha_2)$-metric as $F=\sqrt{L(\alpha_1^2,\alpha_2^2)}$ for simplicity
of the computation.  Theorem \ref{define-requirement}
also gives the condition for the function $L(u,v)$.

\begin{corollary}
If $F= \sqrt{L(\alpha_1^2,\alpha_2^2)}$ defines an $(\alpha_1,\alpha_2)$-norm, then
$L(u,v)$ is a smooth function on the region $\{ u\geq 0, v\geq 0\}\backslash \{0\}\subset\mathbb{R}^2$.
\end{corollary}

\begin{proof}
Assume that $F=\alpha\phi(\frac{\alpha_2}{\alpha})=\alpha\psi(\frac{\alpha_1}{\alpha})$. Then by Theorem
\ref{define-requirement},   $\phi(s)$ and $\psi(s)$  are smooth  on $[0,1]$.
Since $\phi(s)=\psi(\sqrt{1-s^2})$, they are also smooth functions of $\tilde{s}=s^2$
on $[0,1]$. Now the function $L$ can be expressed in terms of $\phi$ as
$$
L(u,v)=(u+v)\phi^2(\sqrt{\frac{v}{u+v}}).
$$
When $v>0$ and $u\geq 0$, the smoothness of $L$ follows from the smoothness of $\phi$ as
the function of $\tilde{s}=v/(u+v)$. When $u>0$ and $v\geq 0$, we can use $\psi$
to express $L$ and deduce the smoothness of $L$.
\end{proof}

The linear isometry group of a Minkowski norm $(\mathbb{R}^n,F)$ will be denoted as
$L(\mathbb{R}^n,F)$, and its maximal connected subgroup as $L_0(\mathbb{R}^n,F)$.
It is easily seen that the dimension of the group $L_0(\mathbb{R}^n,F)$ reaches the maximum (which is equal to  $\dim \mathrm{SO}(n)$) if and only if $F$ is a Euclidean norm.  In the following we will show that,  if $F$ is a non-Euclidean $(\alpha_1,\alpha_2)$-norm with  dimension decomposition $(n_1,n_2)$, where $n_1\geq n_2>1$,
then $L_0(\mathbb{R}^n,F)$ is equal to the maximal connected proper subgroup $\mathrm{SO}(n_1)\times \mathrm{SO}(n_2)$ of $\mathrm{SO}(n)$.

Let $\mathrm{SO}(\mathbf{V}_1,\alpha)$ and
$\mathrm{SO}(\mathbf{V}_2,\alpha)$
be the maximal connected subgroups of $\mathrm{SO}(\mathbb{R}^n,\alpha)$ which keep all vectors
in $\mathbf{V}_2$ and $\mathbf{V}_1$ invariant, respectively.
Given an $\alpha$-orthogonal base of $\mathbf{V}_1$ and that of $\mathbf{V}_2$,
$\mathrm{SO}(\mathbf{V}_1,\alpha)$ and $\mathrm{SO}(\mathbf{V}_2,\alpha)$ can be naturally identified with the subgroups $\mathrm{SO}(n_1)$ and $\mathrm{SO}(n_2)$ in $\mathrm{SO}(n)$, respectively. Since $F(y)$ is a function of $\alpha_1(y)$ and $\alpha_2(y)$, it is invariant under
the action $\mathrm{SO}(\mathbf{V}_1,\alpha)\times \mathrm{SO}(\mathbf{V}_2,\alpha)$, that is,
\begin{equation}\label{8}
\mathrm{SO}(\mathbf{V}_1,\alpha)\times \mathrm{SO}(\mathbf{V}_2,\alpha)\subset L_0(\mathbb{R}^n,F).
\end{equation}

Conversely, if  (\ref{8}) holds for a Euclidean norm$\alpha$ on
$\mathbb{R}^n$ and an $\alpha$-orthogonal decomposition $\mathbb{R}^n=\mathbf{V}_1\oplus \mathbf{V}_2$,
then any two vectors $y,y'\in \mathbb{R}^n$ with the same $\alpha_1$-values and the
same $\alpha_2$-values belong to the same orbit of the actions of
$\mathrm{SO}(\mathbf{V}_1,\alpha)\times \mathrm{SO}(\mathbf{V}_2,\alpha)$. Therefore  they have the same $F$-values. Hence the $F$-values
are determined uniquely by the $\alpha_1$- and $\alpha_2$-values. Thus   $F$ is an
$(\alpha_1,\alpha_2)$-norm.


If $\mathrm{SO}(\mathbf{V}_1,\alpha)\times \mathrm{SO}(\mathbf{V}_2,\alpha)$ is a proper subgroup of $L_0(\mathbb{R}^n,F)$, i.e., if $\dim L_0(\mathbb{R}^n,F) >\dim \mathrm{SO}(\mathbf{V}_1,\alpha)\times \mathrm{SO}(\mathbf{V}_2,\alpha)$, then we can find an infinitesimal
generator $X$ of $L_0(\mathbb{R}^n,F)$, and two nonzero vectors $v_1\in \mathbf{V}_1$ and $v_2\in \mathbf{V}_2$, such that $X(v_1)=v_2$ and $X(v_2)=-v_1$. Now $X$ generates an $S^1$-action of
rotations in the $2$-dimensional subspace $W$ generated by $v_1$ and $v_2$. Note that the restriction of  $F$ to $W$
 is invariant under the rotations generated by $X$. Hence the restriction $F|_W$   is Euclidean. Then $F$ must be of the form $\sqrt{a\alpha_1^2+b\alpha_2^2}$, where $a$ and $b$ are constants. Therefore  it is
a Euclidean norm on $\mathbb{R}^n$.

On the other hand, if $L_0(\mathbb{R}^n,F)=\mathrm{SO}(\mathbf{V}_1,\alpha)\times \mathrm{SO}(\mathbf{V}_2,\alpha)$, then the
representation of $L_0(\mathbb{R}^n,F)$ on $\mathbb{R}^n$ naturally splits
$\mathbb{R}^n$ into the sum of two irreducible invariant subspaces,
One being $\mathbf{V}_1$ and the other being $\mathbf{V}_2$. Note that in this case the two subspaces $\mathbf{V}_1$ and $\mathbf{V}_2$ are uniquely
determined by $L_0(\mathbb{R}^n,F)$ when $n_1>n_2$. However, if $n_1=n_2$, then one can exchange the subspaces $\mathbf{V}_1$ and $\mathbf{V}_2$.

To summarize, we have the following lemma.
\begin{lemma}\label{lemma-local-isometry}
Let $F$ be a Minkowski norm on $\mathbb{R}^n$, with $n>3$. Then $F$ is an $(\alpha_1,\alpha_2)$-norm with dimension decomposition $(n_1,n_2)$, where $n_1\geq n_2>1$,
if and only if there is a Euclidean norm $\alpha$ on $\mathbb{R}^n$, and an
$\alpha$-orthogonal decomposition $\mathbb{R}^n=\mathbf{V}_1\oplus \mathbf{V}_2$ with $\dim \mathbf{V}_1=n_1$
and $\dim \mathbf{V}_2=n_2$, such that
$\mathrm{SO}(\mathbf{V}_1,\alpha)\times \mathrm{SO}(\mathbf{V}_2,\alpha)\subset L_0(\mathbb{R}^n,F)$.
In this case, the Minkowski norm $F$ is non-Euclidean if and only if
$L_0(\mathbb{R}^n,F)=\mathrm{SO}(\mathbf{V}_1,\alpha)\times \mathrm{SO}(\mathbf{V}_2,\alpha)$.
When $F$ is non-Euclidean and $n_1>n_2$, the subspaces $\mathbf{V}_1$ and $\mathbf{V}_2$ are uniquely determined by $F$.
When $F$ is non-Euclidean and $n_1=n_2$,  the unordered pair $\{\mathbf{V}_1,\mathbf{V}_2\}$ is
uniquely determined by $F$ and there can be an exchange between $\mathbf{V}_1$ and $\mathbf{V}_2$.
\end{lemma}


In general, the representation of an $(\alpha_1,\alpha_2)$-norm is not unique.
But if we require the datum
to be normalized in the following sense, then by
Lemma \ref{lemma-local-isometry}, the representation is  unique in almost all the cases.

\begin{definition}\label{normalized-datum}
Let $F=\alpha\phi(\frac{\alpha_2}{\alpha})$ be an $(\alpha_1,\alpha_2)$-norm
on $\mathbb{R}^n$ with a dimension decomposition $(n_1,n_2)$, where $n_1\geq n_2>1$.
 The datum $(\phi,\alpha,\mathbf{V}_1,\mathbf{V}_2)$  is called normalized if
$\phi(0)=\phi(1)=1$.
\end{definition}

The notion of  normalized
datum can be  defined similarly when we write an $(\alpha_1,\alpha_2)$-norm in the form $F=\alpha\psi(\frac{\alpha_1}{\alpha})$ or $F=f(\alpha_1,\alpha_2)$.

It is easily seen that a datum $(\phi,\alpha,\mathbf{V}_1,\mathbf{V}_2)$ is normalized in the sense of  Definition \ref{normalized-datum} if and only if
$F(y)=\alpha_1(y)$, $\forall y\in \mathbf{V}_1$ and $F(y)=\alpha_2(y)$, $\forall y\in \mathbf{V}_2$.

 If $F$ is non-Euclidean and $n_1=n_2$, then up to a possible exchange, the subspaces $\mathbf{V}_1$ and $\mathbf{V}_2$ are uniquely determined by $F$. The above assertion indicates that $\alpha_1$
and $\alpha_2$ are also uniquely determined by $F$. Thus   $\phi$ is also uniquely determined by $F$.
From this we deduce
the following corollary of Lemma \ref{lemma-local-isometry}.

\begin{corollary}
Let $F$ be a non-Riemannian $(\alpha_1,\alpha_2)$-norm on $\mathbb{R}^n$ with  dimension
decomposition $(n_1,n_2)$, where $n_1\geq n_2>1$.
If $n_1>n_2$, then  $F$ has a unique normalized datum $(\phi,\alpha,\mathbf{V}_1,\mathbf{V}_2)$. However,  if $n_1=n_2$,
then it has exactly two normalized data up to the exchange between $\mathbf{V}_1$ and $\mathbf{V}_2$.
\end{corollary}

\subsection{Globally defined $(\alpha_1,\alpha_2)$-metrics}

We have two  ways to define global $(\alpha_1,\alpha_2)$-metrics. The first
is the general one.
\begin{definition}\label{global-def-1}
Let $F$ be a Finsler  metric on a manifold $M$.
If the restriction of $F$ to any tangent space  is  an $(\alpha_1,\alpha_2)$-norm with  dimension decomposition $(n_1,n_2)$, where $n_1\geq n_2>1$, then $F$ is called a general $(\alpha_1,\alpha_2)$-metric
with  dimension decomposition $(n_1,n_2)$.
\end{definition}

The second is the special one.
\begin{definition}\label{global-def-2}
Let  $F$ be  a Finsler metric on a manifold $M$.  If there is a Riemannian metric $\alpha$ on $M$ and
an $\alpha$-orthogonal bundle decomposition $TM=\mathcal{V}_1\oplus \mathcal{V}_2$, where $\mathcal{V}_1$ and $\mathcal{V}_2$ are $n_1$- and $n_2$-dimensional linear
subbundles ($n_1\geq n_2>1$) respectively, such that  $F(x,y)$ is  a function of
$\alpha(x,y_1)$ and $\alpha(x,y_2)$, $\forall x\in M$ and $y\in T_x M$, where  $y=y_1+y_2$
is the decomposition of $y$ with respect to the bundle decomposition, then  $F$ is call an $(\alpha_1,\alpha_2)$-metric with
dimension decomposition $(n_1,n_2)$.
\end{definition}

For an $(\alpha_1,\alpha_2)$-metric $F$ in the special sense,  there are positive smooth functions
$\phi$ and $\psi$ on $[0,1]$, and functions $\alpha_1$ and $\alpha_2$ on $TM$ similarly defined as in the last subsection, such that
$F=\alpha\phi(\frac{\alpha_2}{\alpha})=\alpha\psi(\frac{\alpha_1}{\alpha})$. In this sense the notion of $(\alpha_1,\alpha_2)$-metrics  can be viewed as a
generalization of  $(\alpha,\beta)$-metrics.

From the above two definitions, one easily sees that an $(\alpha_1,\alpha_2)$-metric in the sense of Definition \ref{global-def-2} (in the special sense) must be a general $(\alpha_1,\alpha_2)$-metric in the sense of Definition \ref{global-def-1}. Note that for an $(\alpha_1,\alpha_2)$-metric $F$ of the general type, the datum of  $F$ may not be able to be globalized. For example, on the Euclidean space $\mathbb{R}^n$, let $(x^1, x^2,\cdots,x^n)$ be
the standard coordinate system and let $(x^1,\cdots, x^n, y^1,\cdots, y^n)$ be the globally defined standard coordinate system of the tangent bundle $T\mathbb{R}^n$. Define a smooth function
$\varphi$ on $\mathbb{R}^n\times\mathbb{R}$ by
$$\varphi (x, s)=1+\varepsilon_1e^{-|x|^2}s+\varepsilon_2e^{-|x|^2}s^2,\quad x\in \mathbb{R}^n, s\in\mathbb{R},$$
where $|x|=\mathop{\sum}_{i=1}^n(x^i)^2$, $x=(x^1,x^2,\cdots,x^n)$, and $\varepsilon_1, \varepsilon_2$ are positive numbers. Now we define a Finsler metric
$F$ on $\mathbb{R}^n$ by
$$F(x, y)=\sqrt{\sum_{i=1}^n(y^i)^2}\,\,\varphi(x, \frac{\sqrt{(y^1)^2+(y^2)^2}}{\sqrt{\mathop{\sum}\limits_{i=1}^n(y^i)^2}}),\quad x\in \mathbb{R}^n, y\in T_x(\mathbb{R}^n).$$
Then it is easily seen that $F$ is a general  $(\alpha_1,\alpha_2)$-metric. But the datum of $F$ can not be globalized, hence $F$ is not a special $(\alpha_1,\alpha_2)$-metric in the sense of
Definition \ref{global-def-2}.
However, if $F$ is a homogeneous $(\alpha_1,\alpha_2)$-metric  in the sense of either Definition \ref{global-def-1} or \ref{global-def-2}, then in almost all the cases we can find good datum of $F$ which
globally defines the metric.

\begin{theorem}\label{theorem-3}
Let $(M,F)$ be a homogeneous non-Riemannian general $(\alpha_1,\alpha_2)$-space, with  dimension decomposition $(n_1,n_2)$, where $n_1\geq n_2>1$.   Suppose $M=G/H$,
 where $G$ is a closed connected transitive  subgroup of $I_0(M,F)$, and $H$  is the isotropy subgroup of $G$ at a fixed $x\in M$. Assume that $H$ is connected. Then we have the following:
\begin{description}
\item{\rm (1)}\quad There exist a positive smooth function $\phi$ on $[0,1]$, with $\phi(0)=\phi(1)=1$, a smooth Riemannian metric $\alpha$ on $M$, and an $\alpha$-orthogonal bundle decomposition  $TM=\mathcal{V}_1\oplus \mathcal{V}_2$ with the given dimension decomposition and corresponding $\alpha_1$ and $\alpha_2$, such that $F=\alpha\phi(\frac{\alpha_2}{\alpha})$. Moreover,   in each tangent space $TM_x$, the triple $(\phi,\alpha|_{TM_x},{\mathcal{V}_1}_{x},{\mathcal{V}_2}_x)$ is a normalized datum of the Minkowski norm $F(x,\cdot)$.
\item{\rm (2)}\quad Let $\rho$ be an isometry  in $I_0(M,F)$. Then for the global datum in (1), we have $\rho_*\mathcal{V}_1=\mathcal{V}_1$, $\rho_*\mathcal{V}_2=\mathcal{V}_2$, $\rho^* \alpha_1=\alpha_1$, $\rho^* \alpha_2=\alpha_2$ and $\rho^* \alpha=\alpha$. A vector field $X$ is a Killing vector
field for $F$ if and only if $L_X \alpha_1=L_X \alpha_2=0$.
\item{\rm (3)}\quad The global datum of $F$ in (1) induces a normalized datum on $\mathfrak{m}$, which is $\mathrm{Ad}(H)$-invariant. The correspondence between the global datum in (1) and the $\mathrm{Ad}(H)$-invariant normalized datum on $\mathfrak{m}$ is one-to-one.
\end{description}
\end{theorem}

\noindent\textbf{Remark}\quad The global datum of $F$ in Theorem \ref{theorem-3} is a good datum of the homogeneous
 metric $F$, and it corresponds to a normalized datum of the $(\alpha_1,\alpha_2)$-norm on each tangent space. We will call it a {\rm good normalized datum}.

\medskip
\begin{proof}
(1)\quad  The Minkowski norm $F(x,\cdot)$ on $T_xM$ must be non-Euclidean, otherwise the homogeneity of $(M,F)$ would imply
that it is a Riemannian metric. Fix a normalized datum $(\phi,\alpha,\mathbf{V}_1,\mathbf{V}_2)$
 for the $(\alpha_1,\alpha_2$-norm $F(x,\cdot)$ on $T_x M=\mathfrak{m}$. By the homogeneity of the space, for any $x'\in M$, there exists $g\in G$
such that $g(x')=x$. Then the datum $(\phi,g^* \alpha, {g^{-1}}^* \mathbf{V}_1,
{g^{-1}}^* \mathbf{V}_2)$ defines a normalized datum for the $(\alpha_1,\alpha_2)$-norm $F(x',\cdot)$ on $T_{x'}M$.
  Let $g'\in G$ be another element satisfying $g'(x')=x$. Then $g^{-1}g'$ belongs to the isotropy group at $x'$, which is conjugate to the connected subgroup $H$. Thus $g^{-1}g'$ induces a linear
isometry on $L_0(TM_{x'},F|{TM_{x'}})$, which acts trivially on the set of normalized data. This means that the normalized data induced by $g$ and $g'$ on $TM_{x'}$ coincide. By the smoothness of $F$, it is easily seen that the set of
the normalized data on  tangent spaces of $M$ defines a  Riemannian metric $\alpha$ on $M$, a smooth $\alpha$-orthogonal bundle decomposition $TM=\mathcal{V}_1\oplus\mathcal{V}_2$,  two smooth functions $\alpha_1$ and $\alpha_2$ on $TM$, and a positive smooth function $\phi$ on $[0,1]$ with $\phi(0)=\phi(1)=1$, such that $F$ can be
globally represented as $F=\alpha\phi(\frac{\alpha_2}{\alpha})$. Hence $F$  is an
$(\alpha_1,\alpha_2)$-metric in the global sense.

(2)\quad Let $(\phi,\alpha,\mathcal{V}_1,\mathcal{V}_2)$ be the global datum of $F$
in (1). Any $\rho\in I_0(M,F)$ can be connected by a continuous path $\rho_t$
in $I_0(M,F)$ such that $\rho_0=\mbox{id}$ and $\rho_1=\rho$. At each $x'\in M$,
the normalized datum induced by $\rho_t^*$ at $x'$, that is, the tiples
$$(\phi,\rho_t^*(\alpha(\rho_t(x),\cdot)),\rho_t^*{\mathcal{V}_1}_{\rho_t(x)},
\rho_t^*{\mathcal{V}_2}_{\rho_t(x)}),$$
defines a continuous path of normalized data at $x'$, which
must be a constant family. Therefore we have  $\rho_*\mathcal{V}_i=\mathcal{V}_i$ and $\rho^*\alpha_i=\alpha_i$, $\forall i=1,2$. Consequently $\rho^*\alpha=\alpha$.

Let $X$ be a  Killing vector field of $F$. Then we have
\begin{equation}\label{9}
L_{X}\alpha_1=L_{X}\alpha_2=0.
\end{equation}
Thus the diffeomorphisms generated by $X$
keep $\alpha_1$,  $\alpha_2$ and $\alpha$ invariant. At the same time these diffeomorphisms induce linear isomorphisms among the null spaces of $\alpha_1$ and $\alpha_2$ in different tangent spaces. Hence they preserves the linear sub-bundles $\mathcal{V}_2$ and $\mathcal{V}_1$. Thus the
diffeomorphisms generated by $X$ are isometries of $F$.

(3)\quad  This follows directly from  the proof of (1).
\end{proof}

Theorem \ref{theorem-3} indicates immediately the existence of a good normalized datum
when $M$ is simply connected. When $M$ is not simply connected, we can use the good datum  of the universal covering manifold to study  local geometric properties.

In another case when $M=G$ is a Lie group and $F$ is a left-invariant non-Riemannian
$(\alpha_1,\alpha_2)$-metric on $G$, the good datum can always be found. In this case the homogeneous space $G$ can be written as a coset space $G'/H$, where
$G'=I_0(G,F)$ and $H$ is the isotropy group of $G'$ at $e\in G$. Since $G'=G/H$ is diffeomorphic
to $G\times H$,  $H$ is connected.

If  in the dimension decomposition  of $F$ we have $n_1>n_2$, then for any $x\in M$,
the normalized datum of the $(\alpha_1,\alpha_2)$-norm $F(x,\cdot)$ on $T_x(M)$   is unique.
In this case, Theorem \ref{theorem-3} holds without the assumption on the connectedness of $H$. The proof only needs some minor changes and will be omitted.

At the end of this section, we give an explicit example of homogeneous
$(\alpha_1,\alpha_2)$-space for which we can find good normalized datum.

Let $G$ be a connected Lie group and $H$ be a compact subgroup of $G$. Suppose the isotropy
representation of $H$ on the tangent space $T_o(G/H)$ at the origin of the coset space $G/H$ can  be decomposed as
\begin{equation}\label{de}
T_o(G/H)=\mathbf{V}_1\oplus \mathbf{V}_2,
\end{equation}
where $\mathbf{V}_1$ and $\mathbf{V}_2$ are irreducible $H$-invariant subspaces of $T_o(G/H)$ with dimensions $\geq 2$. Suppose $\langle,\rangle$ is an $H$-invariant inner product
on $T_o(G/H)$. Then $\langle,\rangle$ can be extended to a $G$-invariant Riemannian metric $\alpha$ on $G/H$ (see \cite{DE12}). On the other hand,
we can define a Minkowski norm on $T_o(G/H)$ by
$$F(X)=\sqrt{\langle X,X\rangle+\sqrt[m]{\langle X_1,X_1\rangle^m+\langle X_2,X_2\rangle^m}},\quad X\in T_o(G/H),$$
where $m\geq 2$ is an integer and $X=X_1+X_2$ is the decomposition of $X$ with respect to (\ref{de}) (see \cite{SZ}). It is easily seen that
$F$ is invariant under the action of $H$. Therefore $F$ can be extended to a $G$-invariant Finsler metric on $G/H$ (see \cite{DE12}), which is obviously a homogeneous $(\alpha_1,\alpha_2)$-metric.

\section{The S-curvatures of  homogeneous $(\alpha_1,\alpha_2)$-spaces}
\subsection{The S-curvature of  homogeneous Finsler spaces}
In \cite{DM6}, we have proven the following formula for the S-curvature of a homogeneous Finsler manifold.

\begin{theorem}\label{S-curvature-formula}
Let $M$ be a homogeneous Finsler space $G/H$, where $H$ is the isotropy group at
$x\in M$. Suppose the Lie algebra $\mathfrak{g}$ of $G$ has a reductive decomposition
 $$\mathfrak{g}=\mathfrak{h}+\mathfrak{m},\quad\mbox{(direct sum of subspaces)}$$
 where $\mathfrak{h}=\mathrm{Lie}\,H$ and $\mathrm{Ad}(h)(\mathfrak{m})\subset \mathfrak{m}$, $\forall h\in H$.
Then for any nonzero $y\in \mathfrak{m}=TM_x$, the S-curvature is given by
\begin{equation}\label{S-F}
S(x,y)=\langle [y,\nabla^{g_{ij}}\ln\sqrt{\det(g_{pq})}(y)]_{\mathfrak{m}},y\rangle_y,
\end{equation}
 where $[\cdot,\cdot]_{\mathfrak{m}}:\mathfrak{m}\otimes\mathfrak{m}\rightarrow\mathfrak{m}$ is the composition of the bracket operation $[\cdot,\cdot]$ with the projection map to $\mathfrak{m}$ with respect to the decomposition $\mathfrak{g}=\mathfrak{h}+\mathfrak{m}$, $\nabla^{g_{ij}}$ is the gradient of the Riemannian metric on $T_xM\backslash 0$ defined by the Hessian matrix $(g_{ij})$, and $\langle\cdot,\cdot\rangle_y$ is the inner product on $T_x(M)$ defined by
the Hessian matrix $(g_{ij}(y))$.
\end{theorem}

For the completeness of the paper,  we  briefly recall the proof in \cite{DM6}.

Given any $x\in M$, one can find     a Killing frame  around $x\in M$, that is, each $X_i$ is a
Killing vector field on an open neighborhood of $x$ and  $X_i|_x$, $i=1,2\ldots,n$, form a basis of $T_x M$. Then
for any nonzero vector $y=y^i X_i(x)\in T_x M$, the geodesic spray $G(x,y)$ is given by
$$
G(x,y)=y^i \tilde{X}_i+\frac{1}{2}g^{il}c^k_{lj}[F^2]_{y^k}y^j \partial_{y^i},
$$
where $\tilde{X}_i$ is a vector field on $TM$ induced by $X_i$, and the coefficients
$c^k_{lj}$s is defined by $[X_l,X_j](x)=c^k_{lj}X_k(x)$. Note that we also have
$[X_l,X_j]_m=-c^k_{lj}X_k$ when $X_i$,  $i=1,2\ldots,n$, are viewed as vectors of the Lie algebra $\mathfrak{g}$.

Since $X_i$,  $i=1,2\ldots,n$ are Killing vector fields,
the derivatives of the distortion function $\tau(x,y)$ vanish in all $\tilde{X}_i$-directions.
To calculate the S-curvature, we need only compute the derivative of $\tau(x,y)$ in the
direction of $\frac{1}{2}g^{il}c^k_{lj}[F^2]_{y^k}y^j \partial_{y^i}$, which gives
\begin{eqnarray*}
g^{il}g^{kh}c^k_{lj}y^h y^j\partial_{y^i}\sqrt{\det(g_{pq})}
=\langle [y,\nabla^{g_{ij}}\ln\sqrt{\det(g_{pq})}(y)]_{\mathfrak{m}},y\rangle_y.
\end{eqnarray*}
From this the formula (\ref{S-F}) follows.

\subsection{The S-curvature of  homogeneous $(\alpha_1,\alpha_2)$-spaces}
We now apply Theorem \ref{S-curvature-formula}
to deduce an explicit S-curvature formula for a non-Riemannian homogeneous
$(\alpha_1,\alpha_2)$-space. The key here is that for  $x\in M$, the connected linear isometry group
$L_0(F(x,\cdot),T_x M)$  provides plenty of rotational symmetries of the norm. These symmetries  imply
that the tangent vector $\nabla^{g_{ij}}\ln\sqrt{\det(g_{pq})}(y)$ of the indicatrix in $T_x M$ is
perpendicular to the $L_0(F(x,\cdot),T_x M)$-orbit through $(x,y)$. Given a  vector $y\in T_x(M)\backslash (\mathbf{V}_1\cup \mathbf{V}_2)$, write  $y=y'+y''$ with respect to the decomposition
$\mathfrak{m}=\mathbf{V}_1\oplus \mathbf{V}_2$. Then  $\nabla^{g_{ij}}\ln\sqrt{\det(g_{pq})}(y)$ is contained in the 2-dimensional space generated by $y'$ and $y''$. This fact will be useful in our computation.
Note that a similar calculation can be carried out
 when $n_2=1$. Hence   the S-curvature formula
we will obtain below also applies  to homogeneous $(\alpha,\beta)$-spaces.

We begin with a good normalized datum
of $(M, F)$. In the case that $M$ is simply connected, the existence of a good normalized datum has been proven in previous sections. If $M$ is not simply connected, we compute for the simply connected covering space of $M$
with the induced homogeneous $(\alpha_1,\alpha_2)$-metric. Since the formula will depends only on the algebraic structure and the metric,
it also applies to $M$.

 Assume that $F$ is  a non-Riemannian
homogeneous $(\alpha_1,\alpha_2)$-metric on $M$  defined by an $\mathrm{Ad}(H)$-invariant
$(\alpha_1,\alpha_2)$-norm on $\mathfrak{m}$ (for the convenience we still denote the norm as $F$). Suppose $(\phi,\alpha,\mathbf{V}_1,\mathbf{V}_2)$ is an $\mathrm{Ad}(H)$-invariant normalized datum  of $F$ on $\mathfrak{m}$.
The inner product induced by $\alpha$ is denoted as $\langle\cdot,\cdot\rangle$.

Let $y\in \mathfrak{m}\backslash
(\mathbf{V}_1\cup \mathbf{V}_2)$.  To calculate  $S(x,y)$, we  choose
the linear coordinates $(y^i)$ with respect to
an $\alpha$-orthonormal basis $\{v_1,\ldots,v_n\}$ of $\mathfrak{m}$,  such that the first $n_1$
vectors form a basis of $\mathbf{V}_1$ and the rest form a basis of $\mathbf{V}_2$. Then we have
\begin{eqnarray*}
\alpha   &=& \sqrt{({y^1})^2+\cdots+({y^n})^2},\\
\alpha_1 &=& \sqrt{({y^1})^2+\cdots+({y^{n_1}})^2},
\end{eqnarray*}
 and
$$\alpha_2 = \sqrt{({y^{n_1+1}})^2+\cdots+({y^n})^2}.$$
We first assume that $y=(a,0,\ldots,0,a')$, with $a>0$, $a'>0$, and that $\alpha(y)=\sqrt{a^2+a'^2}=1$. The projections of $y$ into $\mathbf{V}_1$ and $\mathbf{V}_2$ are denoted
as $y'=(a,0,\ldots,0)$ and $y''=(0,\ldots,0,a')$, respectively.

For the simplicity
of the computation, we write $F$ in the form of $F=\sqrt{L(\alpha_1^2,\alpha_2^2)}$, where
$L$ is positively homogeneous of degree 1. We use $L_1(\cdot,\cdot)$, $L_2(\cdot,\cdot)$, $L_{11}(\cdot,\cdot)$, etc, to denote
the derivatives of $L(\cdot,\cdot)$, with respect to the variables indicated by the lower indices.
Similarly, we use $L_1$, $L_2$, $L_{11}$, etc, to denote their values at $(a^2,a'^2)$. In particular,  we simply write $L(a^2,a'^2)$ as $L$.

In the following we will perform some complicated computations. Some of the calculations, although more involved here, are similar to that of the similar quantities for the Randers case; see \cite{DH13}. First we have
\begin{eqnarray*}
{[F^2]}_{y^i} &=& 2 y^i L_1 (\alpha_1^2,\alpha_2^2),\quad \mbox{ if }i\leq n_1,\\
{[F^2]}_{y^i} &=& 2 y^i L_2 (\alpha_1^2,\alpha_2^2),\quad \mbox{ if }i>n_1,\\
{[F^2]}_{y^i y^i} &=& 2 L_1 (\alpha_1^2,\alpha_2^2)+4 {y^i}^2
L_{11}(\alpha_1^2,\alpha_2^2),\quad\mbox{ if } i\leq n_1,\\
{[F^2]}_{y^i y^i} &=& 2 L_2 (\alpha_1^2,\alpha_2^2)+4 {y^i}^2
L_{22}(\alpha_1^2,\alpha_2^2),\quad \mbox{ if } i> n_1,\\
{[F^2]}_{y^i y^j} &=& 4 y^i y^j L_{11}(\alpha_1^2,\alpha_2^2),\quad
\mbox{ if } i<j\leq n_1\\
{[F^2]}_{y^i y^j} &=& 4 y^i y^j L_{22}(\alpha_1^2,\alpha_2^2),\quad
\mbox{ if } i>j>\leq n_1,\\
{[F^2]}_{y^i y^j} &=& 4 y^i y^j L_{12}(\alpha_1^2,\alpha_2^2),\quad
\mbox{ if } i\leq n_1<j.
\end{eqnarray*}
On the other hand, one easily obtains  the Hessian matrix $(g_{ij}(y))$:
\begin{eqnarray*}
g_{11} &=& L_1+2a^2 L_{11},\\
g_{nn} &=& L_2+2a'^2 L_{22},\\
g_{1n} &=& 2a a' L_{12},\\
g_{ii} &=& L_1,\quad \forall i=1,\ldots,n_1,\\
g_{ii} &=& L_2, \quad\forall i=n_1+1,\ldots,n,
\end{eqnarray*}
with all other $g_{ij}(y)=0$. Furthermore, the inverse matrix $(g^{ij})$ of the Hessian at $y$ is
given by
\begin{eqnarray*}
g^{11}&=& \frac{L_2+2a'^2 L_{22}}{L_1 L_2-2LL_{12}},\\
g^{nn}&=& \frac{L_1+2a^2 L_{11}}{L_1 L_2 -2LL_{12}},\\
g^{1n}&=& \frac{-2aa'L_{12}}{L_1 L_2-2LL_{12}},\\
g^{ii}&=& L_1^{-1}, \quad \forall i=2,\ldots,n_1,\\
g^{ii}&=& L_2^{-1}, \quad \forall i=n_1+1,\ldots,n,
\end{eqnarray*}
with all other $g^{ij}=0$ at $y$.

To determine the coefficients of the mean Cartan torsion, we need first compute the
coefficients of the Cartan tensor. A direct computation shows that
\begin{eqnarray*}
C_{111}(y) &=& 3a L_{11}+ 2a^3 L_{111},\\
C_{nn1}(y) &=& a L_{12} + 2a a'^2 L_{221},\\
C_{n11}(y) &=& a' L_{12} +2 a^2 a'L_{112},\\
C_{nnn}(y) &=& 3 a' L_{22} + 2a'^3 L_{222},\\
C_{ii1}(y) &=& a L_{11}, \quad\forall i=2,\ldots,n_1,\\
C_{iin}(y) &=& a' L_{12},\quad \forall i=2,\ldots,n_1,\\
C_{ii1}(y) &=& a L_{12}, \quad\forall i=n_1+1,\ldots,n,\\
C_{iin}(y) &=& a'L_{22}, \quad\forall i=n_1+1,\ldots,n,
\end{eqnarray*}
and that $C_{ijk}(y)=0$ when  $k\notin \{1,n\}$ and $i=j$,  or $k\notin\{1,n\}$ and $\{i,j\}=\{1,n\}$.
From this we conclude that all the coefficients $I_k=\partial_{y^k}[\ln\sqrt{\det(g_{pq})}]$ of the mean Cartan torsion vanish at $y$ except
\begin{eqnarray*}
I_1 &=&\frac{-LL_{12}-2a^2 LL_{112}}{a(L_1 L_2 -2 LL_{12})}
+(n_1-1)\frac{a L_{11}}{L_1}+(n_2-1)\frac{a L_{12}}{L_2},\\
I_n &=& \frac{-LL_{12}-2a'^2 LL_{122}}{a'(L_1 L_2-2LL_{12})}
+(n_1-1)\frac{a' L_{12}}{L_1}+(n_2-1)\frac{a' L_{22}}{L_2}.
\end{eqnarray*}
Now the vector $\nabla^{g_{ij}}\ln\sqrt{\det(g_{pq})}(y)=g^{ik}(y)[\ln\sqrt{\det(g_{pq})}]_{y^k}(y)$
is the sum of a multiple of $y$ and
\begin{equation}\label{10}
[(\frac{-LL_{12}-2a^2LL_{112}}{aa'(L_1 L_2-2LL_{12})}+
(n_1-1)\frac{aL_{11}}{a'L_1}+(n_2-1)\frac{aL_{12}}{a'L_2}))\frac{L}{a a'(L_1 L_2-L L_{12})}]y'.
\end{equation}
Denote the coefficient of $y'$ in (\ref{10}) as $\Phi(y)$. Then the S-curvature at $y$
is given by
\begin{eqnarray*}
S(x,y)=\Phi(y)\langle[y,y']_{\mathfrak{m}},y\rangle_y=\Phi(y)\langle[y'',y']_{\mathfrak{m}},y\rangle_y.
\end{eqnarray*}
Denote
$$\langle [y'',y']_{\mathfrak{m}},y'\rangle = c,$$
and
$$\langle [y'',y']_{\mathfrak{m}},y''\rangle = d.$$
Then we have
\begin{eqnarray*}
\langle [y'',y']_{\mathfrak{m}},y\rangle_y &=& cg_{11}+d g_{nn}
+(\frac{a'c}{a}+\frac{ad}{a'})g_{1n}\\
&=& L_1 c+L_2 d.
\end{eqnarray*}
Consequently, we get the S-curvature formula for the homogeneous $(\alpha_1,\alpha_2)$-space. We summarize the above as the following theorem.

\begin{theorem}\label{main-theorem-2}
Let $M=G/H$ be a connected simply connected reductive homogeneous manifold with a reductive decomposition of the Lie algebra $\mathfrak{g}=\mathfrak{h}+\mathfrak{m}$. Lat $o$ be the origin of $M$. Identify
the tangent space $T_o(M)$ with $\mathfrak{m}$. Let $F=\sqrt{L(\alpha_1^2,\alpha_2^2)}$ be a non-Riemannian $G$-invariant $(\alpha_1,\alpha_2)$-metric on  $M$ with  dimension decomposition
$(n_1,n_2)$, where $n_1\geq n_2>1$. Suppose $(\phi,\alpha,\mathbf{V}_1,\mathbf{V}_2)$ is an $\mathrm{Ad}(H)$-invariant normalized datum  of $F$ on $\mathfrak{m}$.
Then for any $y\in \mathfrak{m}\backslash (\mathbf{V}_1\cup \mathbf{V}_2)$ with $\alpha(y)=1$, the S-curvature $S(x,y)$ is given by
\begin{equation}\label{main-formula}
S(x,y)=\Phi(y)(L_1(a^2,a'^2)\langle [y'',y']_\mathfrak{m},y'\rangle+L_2(a^2,a'^2) \langle [y'',y']_\mathfrak{m},y''\rangle),
\end{equation}
where $y=y'+y''$ is the decomposition of $y$ with respect to the decomposition $\mathfrak{m}=\mathbf{V}_1\oplus \mathbf{V}_2$, $\alpha(y)=a$, $\alpha(y')=a'$, and{\small
\begin{eqnarray*}
\Phi(y)=\left[(\frac{-LL_{12}-2a^2LL_{112}}{aa'(L_1 L_2-2LL_{12})}+
(n_1-1)\frac{aL_{11}}{a'L_1}+(n_2-1)\frac{aL_{12}}{a'L_2}))\frac{L}{a a'(L_1 L_2-L L_{12})}\right]_{(a^2,a'^2)}.
\end{eqnarray*}}
\end{theorem}
As we mentioned before, the above formula is still valid even if $M$ is not simply connected. Note that in this case, we can use a normalized datum of the universal covering manifold of $M$.

\subsection{Homogeneous $(\alpha_1,\alpha_2)$-metrics with vanishing S-curvature}
We now use the S-curvature formula in Theorem \ref{main-theorem-2} to deduce a necessary and sufficient condition for a homogeneous  $(\alpha_1,\alpha_2)$-metric to have vanishing S-curvature.
\begin{theorem}\label{vanishing-S-curvature-thm}
Keep all the notations as in Theorem \ref{main-theorem-2}.  The S-curvature is everywhere vanishing
 if and only if
\begin{equation}\label{11}
\langle [y',y'']_\mathfrak{m},y'\rangle=\langle [y',y'']_\mathfrak{m},y''\rangle=0,\quad \forall y'\in \mathbf{V}_1, y''\in \mathbf{V}_2.
\end{equation}
\end{theorem}
\begin{proof} That the condition is sufficient follows directly from the formula and  the homogeneity of $F$. We now prove that the condition is necessary.

Suppose the non-Riemannian $(\alpha_1,\alpha_2)$-metric $F$ has vanishing S-curvature but (\ref{11}) does not hold.
Without losing generality, we assume that there are $y'\in \mathbf{V}_1$ and $y''\in \mathbf{V}_2$ such
that $\langle [y',y''],y'\rangle\neq 0$. Denote
$$
a=\langle [y'',y']_\mathfrak{m},y'\rangle \mbox{ and } b=\langle[y'',y']_\mathfrak{m},y''\rangle.
$$
multiplying  $y'$ and $y''$ by suitable scalars if necessary,   we can assume that  $\alpha(y')=\alpha(y'')=1$,
$a>0$ and $b\geq 0$.
Restricted to $y=\sqrt{1-t^2} y'+ty''$,  with $t\in (0,1)$, the coefficients $\Phi$, $L_1$ and $L_2$ in
(\ref{main-formula}) are smooth functions of $t$, which will be denoted as
$\Phi(t)$, $L_1(t)$ and $L_2(t)$ respectively. The same assertion holds  for the S-curvature, which
can be written as
$$
S(t)=t\sqrt{1-t^2}\Phi(t)(a\sqrt{1-t^2}L_1(t)+btL_2(t))\equiv 0.
$$
If we write the metric $F$ as $F=\alpha\phi(\frac{\alpha_2}{\alpha})=\alpha\psi(\frac{\alpha_1}{\alpha})$, then it follows from (\ref{-1}) and (\ref{-2}) that
\begin{eqnarray*}
L_1(t)&=&\phi(t)(\phi(t)-t\phi'(t))>0,
\end{eqnarray*}
and that
\begin{eqnarray*}
L_2(t)&=&\psi(s)(\psi(s)-s\psi'(s))>0\, \mbox{ for } s=\sqrt{1-t^2}.
\end{eqnarray*}
 Therefore we have $\Phi(t)\equiv 0$, $\forall t\in (0,1)$. In the computation of the S-curvature formula, We have seen that $\Phi\equiv 0$ if and only if  the mean Cartan tensor of $M$ is identically $0$. Then the metric $F$ must be Riemannian, which is a contradiction.
\end{proof}

In the special case that $F$ is a left invariant non-Riemannian
$(\alpha_1,\alpha_2)$-metric on a Lie group $G$, the space $\mathfrak{m}$
can be identified with the Lie algebra $\mathfrak{g}$ and
$[\cdot,\cdot]_\mathfrak{m}$ is the Lie bracket of $\mathfrak{g}$. Hence the condition
for $S\equiv 0$ can be stated as the following corollary.

\begin{corollary}\label{corollary-1}
Let $F$ be a left invariant non-Riemannian $(\alpha_1,\alpha_2)$-metric on a Lie group
$G$. Then $S\equiv 0$ if and only if
$$
\langle [y',y''],y'\rangle=\langle [y',y''],y''\rangle=0,\quad \forall y'\in \mathbf{V}_1, y''\in \mathbf{V}_2.
$$
\end{corollary}

\section{Restrictive CW-homogeneity of left invariant  $(\alpha_1,\alpha_2)$-
metrics: Case 1}

In this section we will use  Corollary \ref{corollary-1} to prove Theorem \ref{main-thm-0}.
\subsection{Curvatures of restrictively CW-homogeneous spaces}

In
this subsection we prove an interesting result on S-curvature and flag curvature of restrictive CW-homogeneous Finsler spaces. This result shows that the S-curvature plays an important role in the study
of restrictive CW-homogeneity in Finsler geometry.

\begin{theorem}\label{theorem-4}
A restrictive CW-homogeneous Finsler space has   vanishing S-curvature and non-negative flag curvature.
\end{theorem}
\begin{proof}
First consider the S-curvature. By Proposition \ref{proposition-0},  for any nonzero tangent vector $y\in T_x M$, there is
a KVFCL $Y$ such that $Y(x)=y$. Denote the one-parameter group of isometries generated by $Y$ as $\rho_t$, $t>0$, and the induced diffeomorphisms on $TM$ as $\tilde{\rho}_t$. Then  the function $\tau(x,y)$ in (\ref{-6}) is a constant
along any flow curve of $\tilde{\rho}_t$ in $TM$. Moreover,  the
curves of $\tilde{\rho}_t$ are tangent to the geodesic spray everywhere. Therefore the S-curvature $S(x,y)$, which is the derivative of $\tau(x,y)$ in the direction of the geodesic spray, i.e.,
the direction of the flow curves of $\tilde{\rho}_t$,  must be $0$.

Now we turn to flag curvature. Note that the KVFCL $Y$ is a geodesic field and the Riemannian metric $\hat{g}=(g_{ij}(Y(\cdot)))$
is globally defined. Thus $Y$ is a KVFCL of $\hat{g}$ with the same length as a KVFCL of  $F$.
Let $P$ be a tangent plane in $T_x M$ containing $y$, spanned by $y$ and $u$. Denote by $\hat{K}(P) $
the sectional curvature of $P$ with respect to $\hat{g}$. Then we have
$$
\hat{K}(P)=\frac{\langle \hat{R}_y(u),u\rangle_y}
{\langle y,y\rangle_y\langle u,u\rangle_y-\langle y,u\rangle_y^2}=\frac{\langle \hat{\nabla}_u Y,\hat{\nabla}_u Y\rangle_y}
{\langle y,y\rangle_y\langle u,u\rangle_y-\langle y,u\rangle_y^2}\geq 0,
$$
where $\hat{\nabla}$ is the Levi-Civita connection of $\hat{g}$ (see \cite{BN09}, page 474, Proposition 1). Thus by  Proposition \ref{proposition-1} we have $K(P,y)=\hat{K}(P)\geq 0$.
\end{proof}

\subsection{Proof of Theorem \ref{main-thm-0}}

In this section we will prove  Theorem \ref{main-thm-0} by inducing a contradiction. Suppose  there is a non-Riemannian
left invariant restrictively CW-homogeneous $(\alpha_1,\alpha_2)$-metric $F$ on
a compact connected simple Lie group $G$. Then there exists a good normalized datum of $F$,
which defines a normalized datum $(\phi,\alpha,\mathbf{V}_1,\mathbf{V}_2)$ for the
Minkowski norm on $\mathfrak{m}=\mathfrak{g}$. Since $L(G)\subset I_0(G,F)\subset I_0(G,\alpha)$,
 we have, by \cite{OT76},
$I_0(G,F)\subset L(G)R(G)$. Let $G'$ be the closed connected subgroup of $G$ such that
$R(G')$ is the maximal connected subgroup of right isometric
translations. Then $I_0(G,F)=L(G)R(G')$ and the space of all Killing vector fields
of $(G,F)$ can be identified with  $\mbox{Lie}(I_0(G,F))=\mathfrak{g}
\oplus\mathfrak{g}'$.

The homogeneous space $G$ can be written  as $I_0(G,F)/H$,
where
$$
H=\{L_g R_{g^{-1}}| g\in G'\}
$$
 is isometric to $G'$.
According to Theorem \ref{theorem-3},
there exists a good normalized datum of $F$  which corresponds to the $\mathrm{Ad}(G')$-invariant normalized
datum $(\phi,\alpha,\mathbf{V}_1,\mathbf{V}_2)$ on $\mathfrak{m}=\mathfrak{g}$. On the other hand, $G'$ is   the maximal
connected subgroup of $G$  keeping  $\alpha$, $\mathbf{V}_1$ and $\mathbf{V}_2$ invariant.
Denote the inner product  defined by $\alpha$ on $\mathfrak{g}$  as $\langle\cdot,\cdot\rangle=\langle\cdot,\cdot\rangle_1+\langle\cdot,\cdot\rangle_2$,  where $\langle\cdot,\cdot\rangle_1$ and $\langle\cdot,\cdot\rangle_2$ are the bi-linear
functions defined by $\alpha_1^2$ and $\alpha_2^2$  on $\mathfrak{g}$,  respectively.
Denote the inner product of the bi-invariant metric  as $\langle\cdot,\cdot\rangle_{\mathrm{bi}}$.

From Corollary \ref{corollary-1} and Theorem \ref{theorem-4}, we can get some very clear information about $\alpha$, under the assumption that $\mathbf{V}_2$ is a Cartan subalgebra.

\begin{lemma}\label{lemma-2} Keep all the notations as above and assume that $\mathbf{V}_2$ is a Cartan subalgebra. Then we have the following.
\begin{description}
\item{\rm (1)}\quad The decompositions $\mathfrak{g}=\mathbf{V}_1+\mathbf{V}_2$ and $\mathbf{V}_1=\mathop{\sum}\limits_{\lambda\in\Delta^+}\mathfrak{g}_{\lambda}$,
where $\mathfrak{g}_{\lambda}=\mathfrak{g}\cap(\mathfrak{g}^\mathbb{C}_\lambda+
\mathfrak{g}^\mathbb{C}_{-\lambda})$, and $\Delta^+$ is the set of all positive roots,
are orthogonal with respect to $\alpha$. Restricted to each $\mathfrak{g}_{\lambda}$, $\alpha$ only differs from the bi-invariant metric by a scalar multiplication.
\item{\rm (2)}\quad The space of Killing vector fields of $F$ can be identified with $\mathfrak{g}\oplus \mathbf{V}_2$,
where the first factor corresponds to the left translations and the second factor to
the isometries of right translations.
\end{description}
\end{lemma}

\begin{proof}(1)\quad
We shall actually prove  the assertion under a weaker condition  that $S\equiv 0$. By Corollary \ref{corollary-1},
$\langle [y',y''],y'\rangle=\langle [y',y''],y''\rangle=0$, for any $y'\in \mathbf{V}_1$ and $y''\in \mathbf{V}_2$. Let $\mathbf{V}_2^{\perp_{\mathrm{bi}}}$ be the orthogonal complement of $\mathbf{V}_2$ with respect
to the bi-invariant metric. Then for any regular $y''\in \mathbf{V}_2$, we have
$$
\langle [\mathfrak{g},y''],y''\rangle=\langle [\mathbf{V}_1,y''],y''\rangle=\langle \mathbf{V}_2^{\perp_{bi}},y''\rangle=0.
$$
Thus $\mathbf{V}_2^{\perp_{\mathrm{bi}}}=\mathbf{V}_1$. Note that $\mathbf{V}_1$ is a representation space of the $\mathrm{ad}$-action of the Lie algebra $\mathbf{V}_2$,  and it can
be decomposed as $\mathbf{V}_1=\sum_{\lambda\in\Delta^+}\mathfrak{g}_{\lambda}$, where each subspace is an irreducible $\mathrm{ad}(\mathbf{V}_2)$-invariant subspace. Denote by $K$ the connected Lie subgroup generated by $\mathbf{V}_2$.
 Since $\langle [y',y''],y'\rangle=0$, the $\mathrm{Ad} (K)$-action of $K$ on $\mathbf{V}_1$ is orthogonal with respect to both $\alpha$ and the bi-invariant metric. So the above decomposition of $\mathbf{V}_1$
is orthogonal with respect to both $\alpha$ and the bi-invariant metric. By Schur's Lemma, for any $\beta\in \Delta^+$, there exists a positive scalar $c_{\beta}$,  such that on the subspace $\mathfrak{g}_{ \beta}$, we have $\langle \cdot ,\cdot \rangle=c_{ \beta}\langle\cdot ,\cdot\rangle_{\mathrm{bi}}$.

(2)\quad Since $\mathrm{Ad}(G')\mathbf{V}_2=\mathbf{V}_2$, we have $\mathfrak{g}'\subset \mathbf{V}_2$.
On the other hand,  we have proved in (1) that $\alpha$, $\mathbf{V}_1$ and $\mathbf{V}_2$ are $\mathrm{Ad}$-invariant under the action of the subgroup $K$. Thus  $\mathbf{V}_2\subset\mathfrak{g}'$. Consequently we have $\mathbf{V}_2 = \mathfrak{g}'$.
\end{proof}

Now we consider the restrictive CW-homogeneity. Suppose $F$ is a left invariant restrictively CW-homogeneous $(\alpha_1,\alpha_2)$-metric on $G$. Keep all the notation as above. For any nonzero $X\in \mathbf{V}_1\in TG_e=\mathfrak{g}$, there is a KVFCL
which value at $e$ is $X$. This Killing vector field must be of the form $(X+X',X')$
with  $X'\in \mathbf{V}_2$. Notice that the  value of the Killing vector field at $e\in G$ is the difference of the two components.
Since  $(X+X',X')$ is of constant length,
for any $Y\in \mathbf{V}_1$ and $t\in \mathbb{R}$, we have
\begin{eqnarray*}
L(\alpha_1^2(\mathrm{Ad}(\exp(tY))(X+X')-X'),
\alpha_2^2(\mathrm{Ad}(\exp(tY))(X+X')-X'))
=\mbox{const}>0.
\end{eqnarray*}
Differentiating the above equation   with respect to $t$ and considering the value at $t=0$, we obtain
\begin{equation}
L_1\langle [Y,X+X'],X\rangle_1+L_2\langle [Y,X+X'],X\rangle_2=0,
\end{equation}
 where $L_1$ and $L_2$ are the partial derivatives of $L$ at $(\alpha_1^2(X),0)$ (which  are positive functions). Since $\langle [Y,X+X'],X\rangle_2=0$,
we have
\begin{lemma}\label{lemma-1}
For any $X\in \mathbf{V}_1$, there is a $X'\in \mathbf{V}_2$, such that
\begin{equation}
\langle [Y,X],X\rangle_1=\langle[Y,X'],X\rangle_1, \quad\forall Y\in \mathbf{V}_1.
\end{equation}
\end{lemma}

 For any positive root $\lambda$, select two nonzero vector $Y$ and $Y'$ in $\mathfrak{g}_{\pm\lambda}$, such that $Y$ and $Y'$ are $\alpha$-orthogonal. Furthermore,
for any $X\in \mathbf{V}_1$, let $X'$ be a vector  as in Lemma \ref{lemma-1}.
Then we have
\begin{equation}\label{12}
\langle [Y,X],X\rangle_1 =\langle [Y,X'],X\rangle_1=C(Y, Y')\lambda(X'))\langle Y',X\rangle_1,
\end{equation}
where $C$ is a nonzero function of $Y$ and $Y'$.
The function $f(X)=\langle [Y,X],X\rangle_1$ vanishes on the $\alpha$-orthogonal complement
$Y'^\perp$ of the line generated by $Y'$ in $\mathbf{V}_1$. We now show that $f(X)$
is constantly $0$ on $\mathbf{V}_1$. In fact, given $X=X_1+X_2$, where $X_1$ is a multiple of $Y'$ and $X_2\in Y'^\perp$, we have
\begin{eqnarray*}
\langle [Y,X],X\rangle_1 = \langle [Y,X_1],X\rangle_1+
\langle [Y,X_2],X_2\rangle_1+\langle [Y,X_2],X_1\rangle_1.
\end{eqnarray*}
In the right side of the above equality, the first term is $0$, since  $[Y,Y']\in \mathbf{V}_2$.
On the other hand,  we have just proven that the second term is equal to $0$. Moreover,  the third term is also equal to $0$, since
$$
[Y,X_2]\in [\mathfrak{g}_{\lambda},\sum_{\lambda'\in\Delta^+,\lambda'\neq\lambda}\mathfrak{g}_{\lambda'}]
\subset \sum_{\lambda'\in\Delta^+,\lambda'\neq\lambda}\mathfrak{g}_{\lambda'}\subset Y'^\perp.
$$
Therefore $f$ is equal to $0$ on $\mathbf{V}_1$.

By (\ref{12}),
if $X\notin Y'^\perp$, then $\lambda(X')$ must be $0$. Since this assertions is valid for
any positive roots $\lambda$,  we have $\lambda(X')=0$, $\forall \lambda\in \Delta^+$, that is,  $X'=0$, for all $X$ in  the complement of finite hyperplanes
in $\mathbf{V}_1$, which is an open and dense subset. By the continuity,
all Killing vectors in $\mathfrak{g}\oplus 0$ are KVFCLs. Thus all right
translations of $G$ are isometries. This is a contradiction to (2) of Lemma \ref{lemma-2}. The contradiction comes from the assumption that there exists a left invariant restrictively CW-homogeneous non-Riemannian $(\alpha_1,\alpha_2)$-metric with dimension decomposition $(n_1, n_2)$. This completes the proof of Theorem \ref{main-thm-0}.
\section{Restrictive CW-homogeneity of
left invariant $(\alpha_1,\alpha_2)$-
metrics: Case 2}
\subsection{A key lemma}

The following lemma is crucial for later discussions. Since the proof is rather long, we put it separately in the next section.

\begin{lemma}\label{key lemma}\textbf{\rm (\bf The Key Lemma)}\quad
Let $G$ be a  compact connected simple Lie group with rank $>1$, Lie\,$G=\mathfrak{g}$, and $X$ a nonzero
vector in $\mathfrak{g}$. Then for any nonzero subspace $\mathbf{V}\subset\mathfrak{g}$ with $\dim \mathbf{V}\leq 3$, there
exists  $g\in G$, such that
\begin{equation}\label{aftermerge-0000}
\mathbf{V}\cap \mathrm{Ad}(g) \mathfrak{c}_{\mathfrak{g}}(X)=0.
\end{equation}
\end{lemma}

Note that in Lemma \ref{key lemma}, (\ref{aftermerge-0000}) can
be equivalently stated as
$\mathbf{V}^{\perp_{\mathrm{bi}}}+ [\mathrm{Ad}(g)X,\mathfrak{g}]=\mathfrak{g}$. It is also equivalent to the assertion that
the orthogonal projection with respect to the bi-invariant metric from the $\mathrm{Ad}(G)$-orbit $\mathcal{O}_X$ to $\mathbf{V}$ has
a surjective tangent map somewhere, that is,   the image of the projection contains a non-empty open set
 of $\mathbf{V}$. We thus have the following

\begin{corollary}\label{aftermergekeylemmacorollary}
Let $G$ be a compact connected simple Lie group with rank\,\,$>1$, $\mathrm{Lie}\,(G)=\mathfrak{g}$,
 $X$  be a nonzero vector in $\mathfrak{g}$,
and $\mathbf{V}$ be a linear space with $\dim\mathbf{V}\leq 3$. Suppose $l$ is a
surjective linear map
from $\mathfrak{g}$ onto $\mathbf{V}$. Then  there exists $Y\in \mathcal{O}_X$ such that the restriction of $l$ to the $\mathrm{Ad}(G)$-orbit
$\mathcal{O}_X$ is regular, i.e.,  the tangent map of the restriction $l|_{\mathcal{O}_X}$
is surjective at $Y$.
\end{corollary}

\subsection{A criterion for KVFCLs}

Let $G$ be a compact connected simple Lie group with rank\,\,$>1$,
and $F$ be a non-Riemannian
left invariant $(\alpha_1,\alpha_2)$-metric on $G$, with a decomposition $\mathfrak{g}=T_eG=\mathbf{V}_1\oplus \mathbf{V}_2$, such that
$\mathbf{V}_2$ is a commutative subalgebra of $G$. Then by
Theorem \ref{theorem-3}, there exists
a good normalized datum of $F$, which defines a normalized datum
$(\phi,\alpha,\mathbf{V}_1,\mathbf{V}_2)$ for the induced Minkowski
norm on $T_e G$. For simplicity, we will use the same notations to denote the global
datum of $F$ and  the datum of the Minkowski norm on $T_e G$.
We keep the notations $\langle\cdot,\cdot\rangle$, $\langle\cdot,\cdot\rangle_{\mathrm{bi}}$,
$\langle\cdot,\cdot,\rangle_1$ and $\langle\cdot,\cdot\rangle_2$ as in the previous section.

 In the above we have showed that $I_0(G,F)=L(G)R(G')\subset L(G)R(G)$, where $G'$ is the maximal closed connected subgroup of $G$ whose $\mathrm{Ad}$-action  preserves
$\alpha$, $V_1$ and $V_2$. It is obvious that $\dim G'<\dim G$. The space of Killing vector fields of $F$ can be identified with
the Lie algebra $\mathfrak{g}\oplus\mathfrak{g}'$, where $\mathfrak{g}'=\mathrm{Lie}(G')$.

 Recall that if a Killing vector field of a left invariant Randers metric or a
left invariant $(\alpha,\beta)$-metric on the compact connected simple Lie group $G$ is
of constant length,  then we have either $X=0$ or $X'\in\mathfrak{c}(\mathfrak{g}')$ (see \cite{DM3} and \cite{DM5}).
This criterion is the key for our  study on CW-translations and the CW-homogeneity of left invariant $(\alpha,\beta)$-metrics on compact connected simple Lie groups.
Now we generalize this  criterion to $(\alpha_1,\alpha_2)$-metrics, under the assumption that
$\mathbf{V}_2$ is a commutative subalgebra of $\mathfrak g$.

\begin{theorem} \label{theorem-6}
Let $F$ be a left invariant non-Riemannian $(\alpha_1,\alpha_2)$-metric on
a compact connected simple Lie group $G$. With the same notations as above,
assume that the subspace $\mathbf{V}_2\subset \mathfrak{g}$ is a commutative subalgebra with  dimension $n_2>1$. Let $(X,X')\in\mathfrak{g}\oplus\mathfrak{g}'$ be a nonzero vector which defines a KVFCL on $(G, F)$. Then we have either $X=0$ or $X'\in\mathfrak{c}(\mathfrak{g}')$.
\end{theorem}

\begin{proof}
Let $(X,X')$ be a KVFCL of $F$. Then
$$
\alpha(\mathrm{Ad}(g) X-\mathrm{Ad}(g')X')\phi(\frac{\alpha_2(\mathrm{Ad}(g) X-\mathrm{Ad}(g')X')}
{\alpha(\mathrm{Ad}(g) X-\mathrm{Ad}(g')X')})
$$
is a constant function of  $g\in G$ and $g'\in G'$. The Lie algebra $\mathfrak{g}'$ is a subalgebra of the
normalizer of $\mathbf{V}_2$ in $\mathfrak{g}$,  which is equal to  the centralizer
$\mathfrak{c}_{\mathfrak{g}}(\mathbf{V}_2)$. Thus for any fixed  $g$, the function
\begin{eqnarray*}
\alpha_2(\mathrm{Ad}(g) X-\mathrm{Ad}(g')X')&=&\alpha({\rm pr}_2 \mathrm{Ad}(g) X-{\rm pr}_2 \mathrm{Ad}(g')X')\nonumber\\
&=&\alpha({\rm pr}_2 \mathrm{Ad}(g) X-{\rm pr}_2 X')
\end{eqnarray*}
is a constant function of $g'$. By the assumption on the smooth function $\phi$, we
have $\phi(s)-s\phi'(s)>0, \forall s\in [0,1]$. Hence  $\frac{\phi(s)}{s}$ is strictly
decreasing with respect to $s$. Therefore $\alpha(\mathrm{Ad}(g) X-\mathrm{Ad}(g')X')$ is also a constant function of $g'$ for the fixed $g$. Since neither $\alpha(\mathrm{Ad}(g) X)$ nor $\alpha(\mathrm{Ad}(g')X')$  depends on $g'$,
$\langle \mathrm{Ad}(g) X,\mathrm{Ad}(g')X'\rangle$ is a constant function of $g'$.
Thus $\mathrm{Ad}(g) X$ is $\alpha$-orthogonal to all the tangent spaces of the $\mathrm{Ad}(G')$-orbit
$\mathcal{O}_{X';\mathfrak{g'}}\subset\mathfrak{g}'$, whose linear span is the ideal generated by $[X',\mathfrak{g}']$ in $\mathfrak{g}'$. Now letting $g$ change as well, one concludes that all the tangent spaces of the $\mathrm{Ad}(G)$-orbit $\mathcal{O}_{X;\mathfrak{g}}$  are $\alpha$-orthogonal to the ideal
generated by $[X',\mathfrak{g}']$ in $\mathfrak{g}'$. When $X\neq 0$,  the tangent spaces
of the $\mathrm{Ad}(G)$-orbit $\mathcal{O}_{X,\mathfrak{g}}$  span the ideal generated by
$[X,\mathfrak{g}]$,  which by the simplicity is equal to $\mathfrak{g}$. Therefore  the ideal of $\mathfrak{g}'$ generated by
$[X',\mathfrak{g}']$ is 0, i.e.,  $X'\in\mathfrak{c}(\mathfrak{g'})$. This completes the proof.
\end{proof}
\subsection{The decomposition of the set of KVFCLs}

Theorem \ref{theorem-6}  enables us to decompose the set
of KVFCLs into two closed subsets. In the following we denote the set of KVFCLs of
a Finsler metric $F$ as $\mathcal{K}_{F}$, which can be naturally identified with a closed subset of $\mathfrak{g}\oplus\mathfrak{g}'$. Consider a left invariant non-Riemannian  $(\alpha_1,\alpha_2)$-metric $F$ on a compact connected simple Lie group $G$, with  dimension decomposition $(n_1,n_2)$, where $n_1\geq n_2>1$. Assume that the subspace $\mathbf{V}_2$
in the decomposition $\mathfrak{g}=\mathbf{V}_1+\mathbf{V}_2$  is a commutative subalgebra.
We  denote the closure of $\{(X,X')\in\mathcal{K}_F|X\neq 0,X'\in\mathfrak{c}(\mathfrak{g}')\}$ as $\mathcal{K}_{F;1}$, and that of $0\oplus\mathfrak{g}'$ as $\mathcal{K}_{F;2}$. Theorem
\ref{theorem-6} amounts to saying that $\mathcal{K}_F$ is the union of the two closed subsets
$\mathcal{K}_{F;1}$ and $\mathcal{K}_{F;2}$.
In the following we shall  show that in many cases we have $\mathcal{K}_{F;2}\cap\mathcal{K}_{F;2}=\{0\}$.

\begin{lemma}\label{lemma-7}
 Let $F$ be a left invariant non-Riemannian $(\alpha_1,\alpha_2)$-metric on a compact connected simple Lie group $G$ as  above. Then there is a constant $C>0$, such that for any $(X,X')\in\mathcal{K}_{F;1}$, we have $||X'||_{bi}\leq C||X||_{bi}$,
where $||\cdot||_{\mathrm{bi}}$ is the norm of a bi-invariant inner product $\langle\cdot,\cdot\rangle$ on $\mathfrak{g}$.
\end{lemma}

\begin{proof}
In the proof, the Lie algebra $\mathfrak{g}$ will be viewed as a flat manifold with the metric $\langle\cdot,\cdot\rangle_{\mathrm{bi}}$, and any submanifold in it will be endowed
with the induced metric.

Suppose conversely that the constant $C>0$ indicated in the lemma does not exist. Then
there is a sequence of $(X_n,X'_n)\in\mathcal{K}_{F;1}$ such that
$||X_n||_{\mathrm{bi}}=1$, $X'_n\in\mathfrak{c}(\mathfrak{g}')$ and
$\mathop{\lim}\limits_{n\rightarrow\infty}||X'_n||_{\mathrm{bi}}=\infty$.
Denote $F(\mathrm{Ad}(g)X_n-X'_n)=l_n$. Then the sequence $\{l_n\}$ also diverges to $\infty$.
The $\mathrm{Ad}(G)$-orbit $\mathcal{O}_{X_n}$ is contained in
the hypersurface
\begin{equation}
\mathcal{S}_n=\{Y|F(Y-X'_n)=l_n\}\subset \mathfrak{g},
\end{equation}
on which the $C^0$-norm of all
principal curvatures converges to $0$ when $n\rightarrow \infty$. Taking  a suitable sequence if necessary,
we can assume that  $\mathop{\lim}\limits_{n\rightarrow\infty}X_n=X$. Let $B(3)$ be  the closed round ball with center $0$
and radius $3$ (with respect to the bi-invariant metric), then the intersectional set $B(3)\cap \mathcal{S}_n$ converges to a closed set of the form $B(3)\cap \mathcal{S}$, where $\mathfrak{S}$ is a flat hyperplane $\mathcal{S}$ of codimension $1$
in $\mathfrak{g}$. This implies that the hyperplane $\mathcal{S}$ contains the $\mathrm{Ad}(G)$-orbit
$\mathcal{O}_X$ of the nonzero vector $X$. Thus $\mathfrak{g}$ has a nontrivial ideal, contradicting to the assumption that  $G$ is a simple Lie group.
\end{proof}

The KVFCLs in $\mathcal{K}_{F;2}$ or the CW-translations generated by them are relevant to the (connected) isometry group rather than the metric $F$. Therefore such kind of Killing vector fields  are of  little interest to our study. In most  cases, we need only consider the Killing vector fields in $\mathcal{K}_{F;1}\backslash \{0\}$ and the corresponding CW-translations, as implied by the following corollary of Lemma \ref{lemma-7}.

\begin{corollary}\label{aftermergecorollary-0}
Let $F$ be a left invariant non-Riemannian $(\alpha_1,\alpha_2)$-metric on a compact connected simple Lie group $G$. Then $F$ is restrictively CW-homogeneous if and only if any nonzero tangent vector can be extended to a Killing vector field
$(X,X')\in\mathcal{K}_{F;1}$ with $X\neq 0$.
\end{corollary}

\begin{proof} We just need to prove the ``only if" part.  Since $F$ is homogeneous, we need only prove the assertion for vectors in $T_e(G)$. Now suppose $F$ is restrictively CW-homogeneous and $v\in T_e(G)\backslash \{0\}$. Then there exists a Killing vector field
$(X, X')\in \mathcal{K}_F$ such that the value of $(X, X')$ at $e$ is $v$. Note that $(X, X')\in \mathcal{K}_{F;1}$ if $X\ne 0$. If $X=0$, then we have $v\in \mathfrak{g}'$. As $\dim \mathfrak{g}'<\dim \mathfrak{g}$,
there is a sequence of tangent vectors $v_n\in T_e G$  with $v_n\notin \mathfrak{g}'$,
$\forall n$, such that $v=\mathop{\lim}\limits_{n\to\infty}v_n$. By the above argument, each vector $v_n$ can be extended to a sequence of nonzero
Killing vector fields contained in $\mathcal{K}_{F;1}$. Using a diagonal argument, one can find
a sequence of Killing vector fields $\{w_n\}$, such that  the limit $w=\mathop{\lim}\limits_{n\to\infty}w_n$ exits. Then
  $w$ is a nonzero KVFCL and the value of $w$ at $e$ is $v$. Since $\mathcal{K}_{F:1}$ is a closed subset and $w$ is the limit of a sequence in $\mathcal{K}_{F;1}$, $w$  can be represented as $(X,X')\in\mathcal{K}_{F;1}$ with $X\neq 0$.  This completes the proof of the corollary.
\end{proof}

In the special case of $n_2=2$,
we have the following theorem.

\begin{theorem}\label{theorem-7}
Let $F$ be a left invariant non-Riemannian $(\alpha_1,\alpha_2)$-metric on a compact
connected simple Lie group $G$, with dimension decomposition $(n_1,n_2)$, where $n_1> n_2=2$. With the same notations as above, assume that the subspace
$\mathbf{V}_2\subset\mathfrak{g}$ is a commutative subalgebra. Then for any   $(X,X')\in \mathcal{K}_{F;1}$ and $(X,X'')\in \mathcal{K}_{F;1}$, we have $X'-X''\in \mathbf{V}_2$.
In particular, when $X=0$, both $X'$ and $X''$ must be 0, hence
$\mathcal{K}_{F;1}\cap\mathcal{K}_{F;2}=\{0\}$.
\end{theorem}

\begin{proof}
If $X=0$,   then by Lemma \ref{lemma-7} we have $X'=X''=0$.
Thus $\mathcal{K}_{F;1}\cap\mathcal{K}_{F;2}=\{0\}$.
Therefore we need only consider the case $X\ne 0$.
Let $(X,X')$ be a KVFCL of $F$ with $X\neq 0$.
Write $F$ as $F=\sqrt{L(\alpha_1^2,\alpha_2^2)}$.
Then
\begin{equation}\label{aftermerge0}
L(\alpha_1^2(Ad(\exp(tY)g)X-X'),\alpha_2^2(Ad(\exp(tY)g)X-X'))
\end{equation}
is a constant function of $g\in G$ and $Y\in\mathfrak{g}$.
 Denote $u=\alpha_1^2(\mathrm{Ad}(g) X-X')$ and $v=\alpha_2^2(\mathrm{Ad}(g) X-X')$. Taking the partial derivative of (\ref{aftermerge0}) with respect to $t$ and considering the value at  $t=0$,  we have
{\small\begin{equation}\label{aftermerge1}
L_1(u,v)\langle [Y,\mathrm{Ad}(g) X],\mathrm{Ad}(g) X-X'\rangle_1
+L_2(u,v)\langle [Y,\mathrm{Ad}(g) X],\mathrm{Ad}(g) X-X'\rangle_2=0,
\end{equation}}
where $L_1(\cdot,\cdot)$ and $L_2(\cdot,\cdot)$ are the partial derivatives of $L$,
which are  positive everywhere.
Note that for $[Y,\mathrm{Ad}(g) X]\in \mathbf{V}_1$, we have $\langle [Y,\mathrm{Ad}(g) X],\mathrm{Ad}(g) X-X'\rangle_2=0$. Then by (\ref{aftermerge1}) we have
\begin{equation} \label{aftermerge2}
\langle [Y,\mathrm{Ad}(g) X],\mathrm{Ad}(g) X-X'\rangle_1=0,\quad \mbox{when}\,\,[Y, \mathrm{Ad}(g)X]\in \mathbf{V}_1.
\end{equation}

The same argument can also be applied to  $(X,X'')$. Hence for any $g\in G$ and $Y\in\mathfrak{g}$ with $[Y,\mathrm{Ad}(g) X]\in\mathbf{V}_1$, we have
{\small\begin{eqnarray*}
& & \langle [Y,\mathrm{Ad}(g) X], X'-X''\rangle_1 \nonumber\\
& & =\langle [Y,\mathrm{Ad}(g) X],\mathrm{Ad}(g) X-X''\rangle_1-\langle [Y,\mathrm{Ad}(g) X],\mathrm{Ad}(g) X-X'\rangle_1=0.
\end{eqnarray*}}
To complete the proof of the theorem, We need only  prove that the sets $\{Y|Y\in [\mathfrak{g},\mathrm{Ad}(g) X]\cap\mathbf{V}_1\}$, $\forall g\in G$,  linearly
span  $\mathbf{V}_1$. If this is not true, then there is a nonzero vector $V\in\mathbf{V}_1$, such that
$$
V\in ([\mathfrak{g},\mathrm{Ad}(g) X]\cap \mathbf{V}_1)^{\perp_{\mathrm{bi}}}=
[\mathfrak{g},\mathrm{Ad}(g) X]^{\perp_{\mathrm{bi}}}+\mathbf{V}_1^{\perp_{\mathrm{bi}}},\quad \forall g\in G.
$$
Let $\mathbf{U}=\mathbf{V}_1^{\perp_{\mathrm{bi}}}$ be the orthogonal complement of $\mathbf{V}_1$ with
respect to the bi-invariant linear metric on $\mathfrak{g}$. Then
$V$ is contained in
$$
\bigcap_{g\in G}([\mathfrak{g},\mathrm{Ad}(g) X]^{\perp_{\mathrm{bi}}}+\mathbf{V}_1^{\perp_{\mathrm{bi}}})=\bigcap_{g\in G}(\mathrm{Ad}(g) \mathfrak{c}_{\mathfrak{g}}(X)+\mathbf{U}).
$$
Notice that $V\in \mathbf{V}_1$ is not contained in $\mathbf{U}$. Thus for any $g\in G$,
$\mathrm{Ad}(g)\mathfrak{c}_{\mathfrak{g}}(X)=\mathfrak{c}_{\mathfrak{g}}(\mathrm{Ad}(g) X)$ has a nonzero intersection with the subspace $\mathbb{R}V+\mathbf{U}$, whose  dimension is $3$. This is a contradiction to Lemma \ref{key lemma}.
\end{proof}

\subsection{The set $\mathcal{K}_{F;1}$}

We keep all notations as above, and further assume that $\mathbf{V}_2$ is
a $2$-dimensional commutative subalgebra of $\mathfrak{g}$ and that $F$ is restrictively
CW-homogeneous.

\begin{lemma}\label{key-lemma-2}
Let $F=\alpha\phi ({\alpha_2}/{\alpha})$ be a CW-homogeneous left invariant non-Riemannian $(\alpha_1,\alpha_2)$-metric on a compact
connected simple Lie group $G$, with dimension decomposition $(n_1,n_2)$, where $n_1> n_2=2$.  Assume that the subspace
$\mathbf{V}_2\subset\mathfrak{g}$ is a commutative subalgebra. Then we have
\begin{description}
\item{\rm (1)}\quad The function $\phi$ is  real analytic  on $[0,1]$.
\item{\rm (2)}\quad The subset $\mathcal{K}_{F;1}\backslash \{0\} $ is a closed real
analytic subvariety of $(\mathfrak{g}\oplus\mathfrak{g}')\backslash \{0\}$.
\item{\rm (3)}\quad For any $X\in\mathfrak{g}$,  there are at most finite many $X'$, which have different $\mathbf{V}_2$-components
 with respect to the decomposition  $\mathfrak{g}=\mathbf{V}_1\oplus\mathbf{V}_2$, such that
$(X,X')\in\mathcal{K}_{F;1}$.
\end{description}
\end{lemma}
\begin{proof}
(1)\quad  Given $s_0\in [0,1]$, there exists a tangent vector $v$ with $F(v)=1$, such that $\frac{\alpha_2(v)}{\alpha(v)}=s_0$. By Corollary \ref{aftermergecorollary-0}, the tangent vector $v$ can
be extended to a Killing vector field $(X,X')$ in $\mathcal{K}_{F;1}$. Then we have
\begin{equation}\label{aftermerge001}
\alpha(\mathrm{Ad}(g)X-X')\phi(\frac{\alpha_2(\mathrm{Ad}(g)X-X')}
{\alpha(\mathrm{Ad}(g)X-X')})=1,\quad \forall g\in G.
\end{equation}
The function
$$s(g)=\frac{\alpha_2(\mathrm{Ad}(g)X-X')}{\alpha(\mathrm{Ad}(g)X-X')}$$
 can not be a constant function
for $g\in G$, otherwise by (\ref{aftermerge001}) $\alpha_2(\mathrm{Ad}(g)X-X')$ would be  a constant
function for $g\in G$, and the $\alpha$-orthogonal projection of the $\mathrm{Ad}(G)$-orbit
$\mathcal{O}_X$ of $X$ in $\mathbf{V}_2$ is contained in an ellipsoid, which is a contradiction
to Corollary \ref{aftermergekeylemmacorollary}. From now on,  we will denote the set $\{s(g)|\,g\in G\}$ as
$\mathcal{I}_{(X,X')}$, which is a closed interval  $[r_0, r_1]\subset [0,1]$.

First suppose $s_0\in (r_0, r_1)$.
We now assert that there is an element $X$ within the orbit $\mathcal{O}_X$, and a vector
$Y\in\mathfrak{g}$ such that $s(X-X')=s_0$, and
the real analytic function $f(t)=s(\exp(tY))$ satisfies the conditions
\begin{equation}\label{fff}
f(0)=s_0,\quad
f'(0)=f''(0)=\cdots=f^{(k-1)}(0)=0,\quad f^{(k)}(0)>0,
\end{equation}
for some $k\in \mathbb{N}$. In fact, it is easily seen that there exists $g\in G$
such that $s(g)>s_0$.  Assume that $g$ belongs to the one-parameter
subgroup $\exp(tY)$  ¡¡generated by $Y\in\mathfrak{g}$. Suppose
conversely that
for both $f_1(t)=s(\exp(tY))$ and $f_2(t)=s(\exp(-tY))$, the first
nonzero derivative $f^{(k)}$, $k>0$, at each $t_0$ with $s(\exp(t_0 Y))=s_0$, is negative.
Then $f(t)=s(\exp(tY))$ reaches a local maximum at each $t_0$
with $s(\exp(t_0 Y)=s_0$. By the mid-value theorem for continuous
functions,  $s_0$ is the maximum of $f(t)$.
But this is a contradiction to the assumption. This prove our assertion.

Now suppose that $f$ is a function satisfying (\ref{fff}). Using a suitable real analytic change of variable $\tilde{t}=\tilde{t}(t)$ with $\tilde{t}(0)=0$, we can assume that
 $f(t)=f(0)+\tilde{t}^k$ on a small neighborhood $O'$ of $0$.
Then the equality (\ref{aftermerge001}) can  be rewritten as
\begin{equation}\label{aftermerge002}
\phi(s_0+\tilde{t}^k)=\frac{1}{\alpha(\mathrm{Ad}(\exp(tY))X-X')}.
\end{equation}
Now on the set $O'$,
the left side of (\ref{aftermerge002}) is a smooth function of $\tilde{t}$, and we have
$$\frac{d^l\phi}{d\tilde{t}^{l}}|_{\tilde{t}=0}=0, \quad \forall l,\,\mbox{with}\,\, k\not |\, l.$$
Thus the right
side of (\ref{aftermerge002}) is a real analytic function of $\bar{t}=f(t)$ at the positive side of $f(0)$. Hence
$\phi(s)$ is a real analytic function of $s$ for  $s\geq s_0$. A similar argument can be used to show that $\phi (s)$ is a real analytic
 function of $s$ for $s\leq s_0$.

Now suppose  $s_0=r_0$ is an endpoint of $\mathcal{I}_{(X,X')}$. The above   argument shows  that $\phi(s)$ is
real analytic for  $s\geq s_0$. We now use Lemma \ref{lemma-7} to prove the real analytic
property of $\phi(s)$ for $s\leq s_0$.
If there is another Killing vector field $(X_0,X'_0)$ in
$\mathcal{K}_{F;1}$ such that an open neighborhood of $s_0$ is contained in $\mathcal{I}_{(X_0,X'_0)}$, then it is done.
Otherwise we can find a sequence $s_n$ approaching $s_0$
from below. For each $s_n$, we can find a KVFCL $(X_n,X'_n)\in\mathcal{K}_{F;1}$
 with length 1, such that
$s_n$ is contained in $\mathcal{I}_{(X_n,X'_n)}$ which lies below $s_0$. Taking a subsequence, this sequence
of KVFCLs converges to a KVFCL $(X_0,X'_0)\in\mathcal{K}_{K;1}$, such that $\mathcal{I}_{(X,X')}$ contains the negative side of the endpoint $s_0$. A similar argument can be applied to the case of $s_0=r_1$. This completes the proof of (1).


(2)\quad  Given $(X_0,X'_0)\in\mathcal{K}_{F;1}$ with $X_0\neq 0$, there is an neighborhood $O$ of $(X_0, X_0')$ in $\mathfrak{g}\oplus\mathfrak{g}'$ such that an element
$(X, X')\in O$ lies in
$\mathcal{K}_{F;1}\backslash \{0\}$ if and only if
$$
\alpha(\mathrm{Ad}(g)X-X')\phi(\frac{\alpha_2(\mathrm{Ad}(g)X-X')}
{\alpha(\mathrm{Ad}(g)X-X')})=\alpha(X-X')\phi(\frac{\alpha_2(X-X')}
{\alpha(X-X')}), \forall g\in G.
$$
These equations  are real analytic  with respect to $X$ and $X'$,  since $\phi$ is real analytic on $[0,1]$.
Thus $\mathcal{K}_{F;1}\backslash\{0\}$ is a closed real analytic subvariety of $(\mathfrak{g}\oplus\mathfrak{g}')\backslash \{0\}$.

(3)\quad We will prove this part by deducing a contradiction.
Suppose conversely that there is an infinite sequence of distinct vectors $\{X'_n\}$, such
that $(X,X'_n)\in\mathcal{K}_{F;1}$. Then Lemma \ref{lemma-7} and Theorem \ref{theorem-7} indicates that $X\neq 0$ and $X'_n$ is
a bounded sequence. By taking subsequence if necessary, we can assume that $X'_n$ converges to a vector $X'$ which is
different from any $X'_n$, and that $Y_n=\frac{(X'-X'_n)}{\alpha(X'-X'_n)}$ converges to a nonzero vector  $Y$. Denote $s(g)=\mathrm{Ad}(g)X-X'$ and
$s_n(g)=\mathrm{Ad}(g)X-X'_n$. Express the metric $F$ as $F=\sqrt{L(\alpha_1^2,\alpha_2^2)}$. Then
$L(\alpha_1^2(s(g)),\alpha_2^2(s(g)))$ and
$L(\alpha_1^2(s_n(g)),\alpha_2^2(s_n(g)))$ are constant functions of $g\in G$.

By Theorem \ref{theorem-7}, $s(g)-s_n(g)\in\mathbf{V}_2$, so $\alpha_1(s(g))=\alpha_1(s_n(g))$. By the differential mid-value theorem,
there is a vector $\xi_n(g)\in\mathfrak{g}$ on the line segment connecting $s(g)$ and $s_n(g)$ such that
\begin{eqnarray}\label{aftermerge003}
&&\frac{1}{\alpha(X'-X'_n)}(L(\alpha_1^2(s(g)),\alpha_2^2(s(g))-L(\alpha_1^2(s_n(g)),\alpha_2^2(s_n(g))))\nonumber\\
&&=2L_2(\alpha_1^2(s(g))),\alpha_2^2(\xi_n(g)))\langle \xi_n(g),Y_n\rangle_2.
\end{eqnarray}
Now the right side of (\ref{aftermerge003}) converges to
\begin{equation}\label{aftermerge004}
2L_2(\alpha_1^2(s(g))),\alpha_2^2(s(g)))\langle s(g),Y\rangle_2
\end{equation}
for each $g\in G$ as $n\rightarrow \infty$. Thus
(\ref{aftermerge004}) is a constant function of $g\in G$.
Therefore
$$
L_2(\alpha_1^2(\mathrm{Ad}(g)X-X'),\alpha_2^2(\mathrm{Ad}(g)X-X'))
\langle \mathrm{Ad}(g)X-X',Y\rangle_2
$$
is a constant function of $g\in G$.

By Corollary \ref{aftermergekeylemmacorollary},  we can replace $X$
with one of its conjugations, such that the $\alpha$-orthogonal projection from
the $\mathrm{Ad}(G)$-orbit $\mathcal{O}_X$ to $\mathbf{V}_2$ has surjective
tangent map at $X$. Let $\mathcal{P}\subset \mathbf{V}_2$ be a hyperplane passing $\mathrm{pr}_2(X-X')$ $\alpha$-orthogonal to $Y$. The set
$$
\{g\in G| \mathrm{pr}_2(\mathrm{Ad}(g)X-X')\subset \mathcal{P}\}
$$
contains a smooth submanifold $\mathcal{S}$ of $G$ around $e$ which is mapped by $\mathrm{pr}_2$
onto a neighborhood of $\mathrm{pr}_2(X-X')$ in $\mathcal{P}$. Restricted to $g\in
\mathcal{S}$, $\langle \mathrm{Ad}(g)X-X',Y\rangle_2$ is a constant function. Then
$$
L_2(\alpha_1^2(\mathrm{Ad}(g)X-X'),\alpha_2^2(\mathrm{Ad}(g)X-X'))
$$
is a constant function of $g\in \mathcal{S}$. We assert that $\alpha_1^2(\mathrm{Ad}(g)X-X')$ and
$\alpha_2^2(\mathrm{Ad}(g)X-X')$ are  linearly independent functions of $g\in \mathcal{S}$. In fact, otherwise
 both are constant functions of $g\in\mathcal{S}$.
If $\alpha_2^2(\mathrm{Ad}(g)X-X')=\mbox{const}$, $\forall g\in \mathcal{S}$,
then the projection $\mathrm{pr}_2$ maps $\mathcal{S}$ to
an ellipsoid, or a point, in $\mathbf{V}_2$,
which is a contradiction to the fact that it maps $\mathcal{S}$ onto an open set of a flat hyperplane in $\mathbf{V}_2$.
This proves the assertion.
Now $L_2$ is positively homogeneous of degree $0$, so it is a constant function on a nonempty
open cone, where $L$ is a linear function. Since $\phi$ is analytic, $L$ is also analytic.
 Therefore $L(u,v)$ must be the same linear function on the whole quarter plane. Thus
 $F$ is a Riemannian metric, which is a contradiction.
\end{proof}

There are two natural projections from $\mathcal{K}_{1,F}\backslash \{0\}$ to $\mathfrak{g}\backslash \{0\}$, namely,
$$
\pi_1(X,X')=X-X', \quad \pi_2(X,X')=X.
$$
Note that $\pi_1$ is just the map from a Killing vector field to its value at $e$.
Corollary \ref{aftermergecorollary-0} indicates that $\pi_1$ is surjective, and (3) of Lemma \ref{key-lemma-2} indicates that $\pi_2$ is a finite covering map. By the locally finite stratification given by Whitney \cite{WH}, there is an open subset $\mathcal{V}'$ of $\mathcal{K}_{F;1}\backslash \{0\}$, which is a smooth manifold with the same dimension as $\mathcal{K}_{F;1}\backslash \{0\}$, such that the restriction of $\pi_2$ on $\mathcal{V}'$ is a finite map. So the manifold $\mathcal{V}'$ has the same dimension as $\mathfrak{g}$, and there is point $p$  in $\mathcal{V}'$ such that $\pi_2$ is regular on a neighborhood of $p$. Thus the real analytic subvariety
$\mathcal{K}_{F;1}\backslash \{0\}$ has the same dimension as $\mathfrak{g}$ and the image
$\pi_2(\mathcal{K}_{F;1}\backslash \{0\})$ contains a nonempty open subset $\mathcal{U}$ in $\mathfrak{g}\backslash \{0\}$. The $\mathrm{Ad}(G)$-actions
on the first factor preserve $\mathcal{K}_{F;1}\backslash \{0\}$.
So we can assume that $\mathcal{U}$ is
$\mathrm{Ad}(G)$-invariant.

Let $\mathfrak{t}$ be a Cartan subalgebra of $\mathfrak{g}$. Then
$\mathcal{U}'=\mathcal{U}\cap(\mathfrak{t}\backslash \{0\})$ is a nonempty open subset of $\mathfrak{t}$.
For any nonzero $X$ in $\mathcal{U}'$, there is a $X'\in\mathfrak{c}(\mathfrak{g}')$, such that
$(X,X')\in \mathcal{K}_{F;1}$. Note that  there maybe many choices for $X'$, but the projections
$\mathrm{pr}_1 X'$ are all equal.

From (\ref{aftermerge2}), we have seen that the map: $l(X)=\mathrm{pr}_1 X'$, $X\in\mathcal{U}'$, satisfies
the following condition
\begin{equation}\label{aftermerge010}
\langle [Y,\mathrm{Ad}(g)X],\mathrm{Ad}(g)X-l(X)\rangle_1=0,\mbox{ whenever }[Y,\mathrm{Ad}(g)X]\in\mathbf{V}_1.
\end{equation}

Now we will show that $l(X)$ can be extended to a linear map on $\mathfrak{t}$ with
(\ref{aftermerge010}) satisfied.

Let $\{X_1,\ldots,X_m\}$ be  a basis of $\mathfrak{t}$  such that any $X_i$ is a  regular vector in $\mathcal{U}'$. By (\ref{aftermerge010}), for each $X_i$, there is $X''_i=l(X_i)$ such that
$
\langle [Y,\mathrm{Ad}(g)X_i],\mathrm{Ad}(g)X_i-X''_i\rangle_1=0,
$
whenever $[Y,\mathrm{Ad}(g)X_i]\in\mathbf{V}_1$.
For $X=\sum_{i=1}^m c_i X_i$, let $X''=\sum_{i=1}^m c_i X''_i$.
Since $[\mathfrak{g},\mathrm{Ad}(g)X]\subset [\mathfrak{g},\mathrm{Ad}(g)X_i]$, $\forall i$,
we have
\begin{equation}\label{aftermerge008}
\langle [Y,\mathrm{Ad}(g)X]\cap\mathbf{V}_1,\mathrm{Ad}(g)X_i-X''_i\rangle_1=0.
\end{equation}
Take the linear combination of (\ref{aftermerge008}) for all $i$,
we get
\begin{equation}\label{aftermerge009}
\langle [Y,\mathrm{Ad}(g)X]\cap\mathbf{V}_1,\mathrm{Ad}(g)X-X''\rangle_1=0.
\end{equation}
This defines a linear map from $X$ to $X''$, satisfying (\ref{aftermerge010}).
From the proof of Theorem \ref{theorem-7}, we easily see that this linear map
coincide with $\mathrm{pr}_1(X')$ when $(X,X')\in\mathcal{K}_{F;1}$.

Now  for any $X_1, X_2\in \mathrm{pr}_1 \mathcal{K}_{F;1}\cap \mathfrak{t}$ in the same orbit of Weyl group actions,
there exists  $X'$s such that $(X_1,X'),(X_2,X')\in \mathcal{K}_{F;1}$. Thus $l(X_1)=l(X_2)$, that is,
 the linear map $l$ on $\mathfrak{t}$ is invariant under the action of the Weyl group. Hence $l=0$.

Since the above assertion is valid for any Cartan subalgebra $\mathfrak{t}$,  we have
\begin{lemma}\label{aftermergelemma-008}
 For any $(X,X')\in\mathcal{K}_{F;1}$,
we have $X'\in \mathbf{V}_2\cap\mathfrak{c}(\mathfrak{g}')$.
\end{lemma}

\subsection{Proof of Theorem \ref{main}}
Keep all notations as in the last subsection. We now show that
 the properties of $\mathcal{K}_{F;1}$ can be used to determine the metric $\alpha$.
Then we give a proof of  Theorem \ref{main}.

Given nonzero $X\in\mathcal{U}$, we can find a pair $(X,X')\in\mathcal{K}_{F;1}$. We have
just proven that $X'\in\mathbf{V}_2\cap\mathfrak{c}(\mathfrak{g}')$. Applying the
equality (\ref{aftermerge1}) to $g=e$, one easily sees that there exists $X''\in\mathbf{V}_2$,  such that
\begin{equation}\label{aftermerge013}
\langle[Y,X], X-X''\rangle=0,\quad \forall Y\in\mathfrak{g}.
\end{equation}
In fact, one just needs to take
$$
X''=\mathrm{pr}_2((-\frac{L_2(u,v)}{L_1(u,v)}+1)X)+\frac{L_2(u,v)}{L_1(u,v)}X',
$$
where $u=\alpha_1^2(X-X')$ and $v=\alpha_2^2(X-X')$.

The next lemma indicates that (\ref{aftermerge013}) is actually true for all $X\in\mathfrak{g}$. The proof is similar to that of Lemma \ref{aftermergelemma-008}.
\begin{lemma}\label{aftermergelemma010}
Keep all the notations as above. For any $X\in\mathfrak{g}$,
there exists $X''\in\mathbf{V}_2$ such that
\begin{equation}\label{aftermerge015}
\langle [Y,X],X-X''\rangle=0,\quad \forall Y\in\mathfrak{g}.
\end{equation}
\end{lemma}

\begin{proof}
Given $X\in\mathfrak{g}$, let $\mathfrak{t}$ be a Cartan subalgebra containing $X$,
and $\{X_1,\ldots,X_m\}$ a basis of $\mathfrak{t}$, such that each $X_i$ is a regular
vector in $\mathcal{U}\cap\mathfrak{t}$. We have proven that for each $X_i$, there is an  $X''_i\in\mathbf{V}_2$
such that the pair $X_i, X_i''$ satisfy (\ref{aftermerge015}). Now given an arbitrary  $X=\sum_{i=1}^m c_iX_i$,  set
$X''=\sum_{i=1}^mc_i X''_i$. Since for each $i$,  $[\mathfrak{g},X]\subset [\mathfrak{g},X_i]$,
it is easily seen that $X, X''$ satisfy (\ref{aftermerge015}).
\end{proof}

Next We shall show that, in the above lemma,  there is a linear map from $\mathfrak{g}$ to $\mathbf{V}_2$ such that  $X''$ is the image of  $X$ under this map.

Let $l_0:\mathfrak{g}\rightarrow\mathfrak{g}$ be the linear isomorphism defined by
$\langle X,Y\rangle=\langle X,l_0(Y) \rangle_{\mathrm{bi}}$.
%
For our purpose it will be important to  choose a suitable basis of $\mathbf{U}=l_0(\mathbf{V}_2)$. We need the following lemma.
\begin{lemma}\label{lemma-9}
Let $\mathfrak{g}$ be a compact simple Lie algebra, and $\mathbf{U}$ be a
$2$-dimensional subspace of $\mathfrak{g}$. Then there is a basis $\{U_1,U_2\}$
of $\mathbf{U}$,
such that there are no containing relations between the centralizers of $U_1$ and $U_2$,
i.e., there are vectors $Y_1\in\mathfrak{c}_{\mathfrak{g}}(U_1)$ and
$Y_2\in\mathfrak{c}_{\mathfrak{g}}(U_2)$, such that $[U_1,Y_2]\neq 0$ and
$[U_2,Y_1]\neq 0$.
\end{lemma}
\begin{proof} If $\mathbf{U}$ is not commutative, then the assertion is obvious. If $\mathbf{U}$ is commutative,
then there is a Cartan subalgebra $\mathfrak{t}$ containing
$\mathbf{U}$. Denote by $\Delta_1$ (resp. $\Delta _2$) the root system of $\mathfrak{g}^\mathbb{C}$ (resp. $\mathfrak{c}_{\mathfrak{g}^\mathbb{C}}(\mathbf{U})$) with respect to $\mathfrak{t}^\mathbb{C}$.
Then for any  $\eta\in \Delta_1\backslash \Delta_2$, we have
    $\dim \mathcal{P}_\eta\cap \mathbf{U}= 1$, where $\mathcal{P}_\eta\subset\mathfrak{t}$ is the Weyl wall of $\eta$. We assert that there must be two roots
    $\eta_1, \eta_2\in\Delta_1\backslash \Delta_2$ such that $ \mathcal{P}_{\eta_1}\cap \mathbf{U}\ne \mathcal{P}_{\eta_2}\cap \mathbf{U}$. In fact, otherwise there exists $U\neq 0$ which spans  this common intersection. Then we have
$U\in\mathfrak{c}(\mathfrak{g})$, which is a contradiction. This proves our assertion.
Now let $\eta_1$ and $\eta_2$ be  two roots such that
$\mathcal{P}_{\eta_1}\cap\mathbf{U}\neq \mathcal{P}_{\eta_2}\cap\mathbf{U}$.  Then there exist nonzero vectors $U_i$, $i=1,2$,  in
$\mathcal{P}_{\eta_i}\cap\mathbf{U}$ and nonzero vectors
$Y_i$, $i=1,2$, in $\mathfrak{g}_{
\eta_i}=\mathfrak{g}\cap(\mathfrak{g}^\mathbb{C}_{\eta_i}+
\mathfrak{g}^\mathbb{C}_{-\eta_i})$. Then we have $Y_1\in\mathfrak{c}_{\mathfrak{g}}(U_1)$,
$Y_2\in\mathfrak{c}_{\mathfrak{g}}(U_2)$ and  $[U_1,Y_2]\neq 0$,
$[U_2,Y_1]\neq 0$.
\end{proof}

Let $\{U_1,U_2\}$ be a basis of $l_0(\mathbf{V}_2)$ as in the above lemma, with the
corresponding vectors $Y_1$ and $Y_2$. For any $X\in\mathfrak{g}$,
let $X''$ be a vector satisfying the condition of Lemma \ref{aftermergelemma010}. Denote $l_0 (X'')=c_1 U_1+c_2 U_2$.
Then by  (\ref{aftermerge015}), we have
\begin{eqnarray*}
\langle X,[X,Y]\rangle &=& \langle l_0 (X''),[X,Y]\rangle_{\mathrm{bi}}  \\
&=& \langle c_1 [Y,U_1]+c_2 [Y,U_2],X\rangle_{\mathrm{bi}},
\end{eqnarray*}
for all $Y\in\mathfrak{g}$.

The function $f_1(X)=\langle X,[X,Y_2]\rangle=c_1
\langle [Y_2,U_1],X\rangle_{\mathrm{bi}}$
vanishes at
$$X\in\{Z| \langle [Y_2,U_1],Z\rangle_{\mathrm{bi}}=0\},$$
the later being  a linear subspace of $\mathfrak{g}$ with  co-dimension $1$. Thus $f_1(X)$ can
be decomposed into the product of the linear
function
$\langle [Y_2,U_1],X\rangle_{\mathrm{bi}}$, and another linear function
$\tilde{c}_1(X)$ which is  equal to $c_1$ at any
$X\in\mathfrak{g}\backslash\{Z|\langle [Y_2,U_1],Z\rangle_{\mathrm{bi}}=0\}$.
Replacing $Y_1$ with $Y_2$, we can find a linear function $\tilde{c}_2(X)$ which equals
$c_2$ at $X\in \mathfrak{g}\backslash\{Z|\langle [Y_1,U_2],Z\rangle_{\mathrm{bi}}\}$.

The above argument shows that for any
$X$ in the  open dense subset
$$
\mathfrak{g}\backslash\{Z\in\mathfrak{g}|[Y_2,U_1],Z\rangle_{\mathrm{bi}}=0 \mbox{ or }[Y_1,U_2],Z\rangle_{\mathrm{bi}}=0\},
$$
 the linear map $l'(X)=l_0^{-1}(\tilde{c}_1(X)U_1+\tilde{c}_2(X)U_2)$
satisfies  the equation
\begin{equation}\label{aftermerge-0002}
\langle[Y,X],X-l'(X)\rangle=0,\quad \forall Y\in\mathfrak{g}.
\end{equation}
Therefore (\ref{aftermerge-0002}) is valid for all $X\in\mathfrak{g}$, and
 we get the following refinement of Lemma \ref{aftermergelemma010}.

\begin{lemma}\label{aftermerge011}
There is a linear map $l':\mathfrak{g}\rightarrow\mathbf{V}_2$, such that
$$\langle [Y,X],X\rangle=\langle [Y,X],l'(X)\rangle,\quad  \forall X,Y\in\mathfrak{g}.$$
\end{lemma}

Now we define a bilinear function $f:\mathfrak{g}\times\mathfrak{g}\rightarrow\mathbb{R}$  by
\begin{equation}\label{aftermerge020}
f(X,Y)=\langle X-l'(X),Y\rangle=\langle l_0(X-l'(X)),Y\rangle_{\mathrm{bi}},
\end{equation}
where $l':\mathfrak{g}\rightarrow\mathbf{V}_2$ is the linear map
in Lemma \ref{aftermerge011}. Given a Cartan subalgebra $\mathfrak{t}$,  the linear map $l_1(X)=l_0(X-l'(X))$ of $\mathfrak{g}$
maps all regular vectors in  $\mathfrak{t}$ to $\mathfrak{t}$, hence $l_1$ keeps
the  Cartan subalgebra $\mathfrak{t}$ invariant.
It is well known that there exists a nonzero vector $X\in\mathfrak{g}$, such that $\mathbb{R}X$ is the intersection
of a finite number of Cartan subalgebras of $\mathfrak{g}$. Then any vector in the $\mathrm{Ad}(G)$-orbit $\mathcal{O}_X$  is an eigenvector of the linear map $l_0(X-l'(X))$. Since $\mathfrak{g}$
is simple,  $l_1$ must be a scalar multiple of the identity map. So $f(X,Y)$ is a bi-invariant
inner product on $\mathfrak{g}$.

Now we can determine the metric $\alpha$ completely.
\begin{lemma}\label{forproof}
The following two assertions hold.
\begin{description}
\item{\rm (1)}\quad  The decomposition $\mathfrak{g}=\mathbf{V}_1\oplus\mathbf{V}_2$ is orthogonal with
respect to the bi-invariant metric. Moreover, there exists two positive numbers $c_1$ and $c_2$ such that for any $Z_i\in \mathbf{V}_i$, $i=1,2$, we have $\alpha_1 (Z_1)=c_1 ||Z_1||_{\mathrm{bi}}$ and $\alpha_2=c_2||Z_2||_{\mathrm{bi}}$, where $||\cdot||_{\mathrm{bi}}$ is the Euclidean norm of $\langle,\rangle_{\mathrm{bi}}$.
\item{\rm (2)}\quad  We have $\mathcal{K}_{F;1}=\mathfrak{g}\oplus 0$.
\end{description}
\end{lemma}

\begin{proof}
Applying (\ref{aftermerge020}) to an arbitrary pair $(X,Y)\in\mathbf{V}_1\times \mathbf{V}_1$, or $(X,Y)\in\mathbf{V}_2 \times\mathbf{V}_1$ and taking into account the fact that $f$ is a bi-invariant inner product, we see immediately that
$\mathbf{V}_1$ and
$\mathbf{V}_2$ are orthogonal with respect to the bi-invariant metric (which is unique up to a positive scalar), and there exists a positive number $c_1$  such that $\alpha_1(Z_1)=c_1 ||Z_1||_{\mathrm{bi}}$, for any $Z_1\in \mathbf{V}_1$. For simplicity, we  assume that $c_1=1$.

Now given any $(X,X')\in\mathcal{K}_{F;1}$ with $X\neq 0$,
we have
\begin{eqnarray}\label{6}
\alpha_1^2(\mathrm{Ad}(g )X-X')&=& ||{\rm pr}_1( \mathrm{Ad}(g) X))||^2_{\mathrm{bi}}\nonumber\\
&=& ||X||_{\mathrm{bi}}^2-||{\rm pr}_2(\mathrm{ Ad}(g) X)||_{\mathrm{bi}}^2.
\end{eqnarray}
By Corollary \ref{aftermergekeylemmacorollary}, up to  a suitable conjugation, we can assume that
the orthogonal projection from the orbit $\mathcal{O}_X$ to $\mathbf{V}_2$
has surjective tangent map at $X$. Therefore there is a codimension $1$ submanifold $\mathcal{N}$ of $G$
near $e$, such that $\mathrm{pr}_2$ maps $\mathrm{Ad}({\mathcal{N}})X$ onto an open submanifold  of the following ellipsoid centered at 0:
\begin{equation}\label{aftermerge025}
\{Y| Y\in\mathbf{V}_2\mbox{ and } ||Y||_{\mathrm{bi}}=||{\rm pr}_2 (\mathrm{Ad}(g) X)||_{\mathrm{bi}}\}.
\end{equation}
 If
$g\in \mathcal{N}$, then by   (\ref{6}),  $ \alpha_1(\mathrm{Ad}(g) X-X')$ is a constant. Thus
$\alpha_2(\mathrm{Ad}(g)X-X')=\alpha_2({\rm pr}_2 \mathrm{Ad}(g) X-X')$ must also be a constant function of    $g\in\mathcal{N}$. This implies that  ${\rm pr}_2 (\mathrm{Ad}({\mathcal{N}}) X)$ contains an open submanifold of an ellipsoid in $\mathbf{V}_2$ centered at $X'$,
namely,
\begin{equation}\label{aftermerge030}
\{Y|Y\in\mathbf{V}_2\mbox{ and } \alpha_2^2(Y-X')=\alpha_2^2(X-X')\}.
\end{equation}
However, the two ellipsoids (\ref{aftermerge025}) and (\ref{aftermerge030}) in $\mathbf{V}_2$
have common open submanifolds only when they coincide. Therefore
$\alpha_2$  is a scalar multiple of the restriction of the bi-invariant metric to $\mathbf{V}_2$. This completes
the proof of (1).

(2)\quad The above argument shows that for any $(X,X')\in
\mathcal{K}_{F;1}$, we have $X'=0$. Thus
$\mathcal{K}_{F;1}=\mathfrak{g}\oplus 0$.
\end{proof}

\textbf{Proof of Theorem \ref{main}}\quad Suppose there exists a left invariant non-Riemannian $(\alpha_1,\alpha_2)$-metric $F$  on a compact connected simple Lie group $G$, with decomposition $\mathfrak{g}=\mathbf{V}_1\oplus \mathbf{V}_2$, such that $\mathbf{V}_2$ is a $2$-dimensional
commutative subalgebra of $\mathbf{G}$. If $F$ is  restrictively CW-homogeneous, then by (2) of Lemma \ref{forproof}, we have $\mathcal{K}_{F;1}=\mathfrak{g}\oplus 0$. This  implies that all the right translations of $G$
are isometries, hence  $I_0(G,F)=L(G)R(G)$. So the $\mathrm{Ad}(G)$-action preserves
$\mathbf{V}_2$, that is,  $\mathbf{V}_2$ is a proper non-zero ideal of $\mathfrak{g}$, which is a
contradiction. This completes the proof of Theorem \ref{main}.

\section{Proof of the Key Lemma}\label{keylemmasection}
In this section we give a proof of Lemma \ref{key lemma}. Here the Lie algebra $\mathfrak{g}$ will always be  endowed with the bi-invariant
metric (which is unqiue up to a positive scalar). We will prove Lemma \ref{key lemma}  by deducing a contradiction.  Suppose conversely that  there exists a nonzero subspace $\mathbf{V}\subset\mathfrak{g}$ with $\dim \mathbf{V}\leq 3$,  such that for any $g\in G$,
\begin{equation}
\mathbf{V}\cap \mathrm{Ad}(g) \mathfrak{c}_{\mathfrak{g}}(X)\ne \{0\}.
\end{equation}

\subsection{The case $\dim\mathbf{V}<3$}

 If $\dim\mathbf{V}=1$, then by the assumption that $\mathbf{V}\cap \mathrm{Ad}(g) \mathfrak{c}_{\mathfrak{g}}(X)\neq \{0\}$, $\forall g\in G$,  we have $\mathbf{V}\subset \mathfrak{c}(\mathfrak{g})$, which
is a contradiction.

Suppose $\dim\mathbf{V}=2$.
Then the minimum of $\dim\mathbf{V}\cap \mathrm{Ad}(g) \mathfrak{c}_{\mathfrak{g}}(X)$,
$ g\in G$ is $1$ or $2$. If it is $2$, then $V$ is contained in the center of
$\mathfrak{g}$, which is a contradiction. So we can suitably change $X$ by
conjugations, such that $\dim\mathbf{V}\cap\mathfrak{c}_\mathfrak{g}(X)=1$.
By the semi-continuity, for all $g\in G$ sufficiently close to $e$, we also have
$\dim\mathbf{V}\cap \mathrm{Ad}(g)\mathfrak{c}_\mathfrak{g}(X)=1$.

Let  $U\in\mathbf{V}$ be a nonzero vector linearly spanning $\mathbf{V}\cap\mathfrak{c}_\mathfrak{g}(X)$, and
 $\{U,U'\}$  a basis of $\mathbf{V}$. Then there is a smooth real function
$f(g)$ of $g\in G$, defined on a small neighborhood $N$ of   $e$, such that $f(e)=0$ and
such that $\mathbf{V}\cap \mathrm{Ad}(g)\mathfrak{c}_\mathfrak{g}(X)$ is linearly spanned by
$U+f(g)U'$.

Setting $g=\exp(tY)$, and taking the differentiation of the equation
$$[U+f(g)U',\mathrm{Ad}(g)X]=0$$
with respect to $t$ at $t=0$, we have
$$
[U,[Y,X]]+Df(Y)[U',X]=0, \quad\forall Y\in\mathfrak{g},
$$
where $Df$ is the differential of $f$ at $e$. Thus $\dim [U,[X,\mathfrak{g}]]\leq 1$.
Since $\dim [U,[X,\mathfrak{g}]]$ is an even number,  it must be 0.
Since $[U,[X,\mathfrak{g}]]=0$
and $[U',X]\neq 0$, we have $Df\equiv 0$.

The above argument on the dimension of the subspace $[U,[X,\mathfrak{g}]]$ is also valid to $\mathrm{Ad}(g)X$, provided
$g\in G$ is sufficiently close to $e$. This implies that there is an neighborhood $N_1$ of $e$ such that $f\equiv 0$ on $N_1$, i.e., $[U,\mathrm{Ad}(g)X]=0$
for any $g\in N_1$. This implies that $U\in\mathfrak{c}
(\mathfrak{g})$, which is a contradiction.

\subsection{The case $\dim\mathbf{V}=3$}

First note that in this case we need only deal with the case that  $\dim\mathbf{V}\cap\mathrm{Ad}(g) \mathfrak{c}_{\mathfrak{g}}(X)=1$ for some $g\in G$. In fact,  otherwise we can choose a proper subspace of $\mathbf{V}$
satisfying  (\ref{aftermerge-0000}), and the proof is then  reduced to the case
$\dim\mathbf{V}<3$.  Replacing  $X$ with certain suitable conjugation if necessary, we can assume that $\dim \mathbf{V}\cap \mathfrak{c}_{\mathfrak{g}}(X)=1$. Let $U\in\mathbf{V}$ be a nonzero vector linearly spanning $\mathbf{V}\cap \mathfrak{c}_{\mathfrak{g}}(X)$,
and  $\{U,U_1,U_2\}$  a basis of $\mathbf{V}$. Then for $g\in G$ sufficiently close to $e$,
we have $\dim V\cap \mathfrak{c}_{\mathfrak{g}}(\mathrm{Ad}(g)X)=1$, and there is a smooth map
$W: G\to \mathfrak{g}$, defined on a small neighborhood $N_2$ of $e$, such that
$W(e)=U$ and  $V\cap \mathfrak{c}_{\mathfrak{g}}(\mathrm{Ad}(g)X)$ is spanned by $W(g)$ for $g\in N_2$.
This means that
there are two smooth functions $f_1(g)$ and $f_2(g)$, defined on  $N_2$,  such that $f_1(e)$=$f_2(e)=0$ and
$[U+f_1(g)U_1+f_2(g)U_2,\mathrm{Ad}(g)X]=0$.
Using a similar argument as in the previous  subsection, we have
\begin{equation}\label{000}
[U,[Y,X]]+Df_1(Y)[U_1,X]+Df_2(Y)[U_2,X]=0,
\end{equation}
where the linear maps $Df_1$, $Df_2:\mathfrak{g}\rightarrow\mathbb{R}$ are the differentials of $f_1$ and $f_2$ at $e$, respectively.
Therefore $\dim [U,[X,\mathfrak{g}]]\leq 3$.  As
 an even number, $\dim[U,[X,\mathfrak{g}]]$ can only be $0$ or $2$.
If $\dim[U,[\mathrm{Ad}(g)X,\mathfrak{g}]]=0$ for all $g\in N_2$, then $f_1=f_2\equiv 0$ on $N_2$, and we can deduce a  contradiction
as in the previous subsection.

The above argument shows that, upon   suitable conjugations,  we can assume that $\dim [U,[X,\mathfrak{g}]]=2$.
By (\ref{000}),
this happens only when $Df_1$ and $Df_2$ are linearly independent.
By the implicit function theorem, any $U'\in\mathbf{V}$ sufficiently close to $U$ spans
the $1$-dimensional $\mathbf{V}\cap\mathrm{Ad}(g)\mathfrak{c}_{\mathfrak{g}}(X)$ for some $g\in N_2$.

To finish the proof of Lemma
\ref{key lemma}, we need only prove that $\dim [U,[X,\mathfrak{g}]]\geq 4$.
This can be equivalently stated as follows. Let $\mathfrak{t}$ be
a Cartan subalgebra of $\mathfrak{g}$ containing $U$ and $X$. Then there are at least $4$ roots, say $\alpha_1,\dots,\alpha_l$, $l\geq 4$, in
the root system $\Delta\subset\mathfrak{t}^*$ of $\mathfrak{g}^C$ with respect to $\mathfrak{t}^C$,  such that $\alpha_i(U)\ne 0$ and $\alpha_i(X)\ne 0$, for all $1\leq i\leq l$. With $\Delta$ viewed as a subset of $\mathfrak{t}$ through
the bi-invariant metric, the above statement is equivalent to   the following assertion.

\medskip
\noindent\textbf{Assertion:}\quad  There are
four roots in $\Delta$ which are not orthogonal to either $X$ or $U$.
\medskip

\subsection{Proof of the Assertion}

The \textbf{Assertion} will be proved by  a case by case argument.

{\bf The case  $\mathfrak{g}=A_n$, $n>1$}. Since $[U,X]=0$, up to  a suitable unitary conjugation, the nonzero vectors $U$ and $X$ can be represented simultaneously as  diagonal matrices. Then it is easy to see that $\dim [U,[X,\mathfrak{g}]]\geq 2$. If the equality holds,  then with a suitable Weyl group action, we can write
$U=a\sqrt{-1}\mbox{diag}(-n,1,\ldots,1)$ and $X=b\sqrt{-1}\mbox{diag}(1,-n,1,\ldots,1)$,
where $a,b\in\mathbb{R}\backslash \{0\}$.

As we argued above, any $U'\in\mathbf{V}$ near $U$ has the same eigenvalue multiplicities as $U$.
The set of all matrices in $\mathfrak{su}(n+1)$ with the same eigenvalue multiplications as $U$
is a smooth manifold, whose tangent space at $U$ is $\mathbb{R}U+[U,\mathfrak{g}]$.
Since $\dim \mathbf{V}>1$, we can find a nonzero vector $V\in\mathbf{V}\cap
[U,\mathfrak{g}]$, which can be written as the matrix
$$
V=\begin{pmatrix}
    0     & -\bar{b}_1 &\cdots & -\bar{b}_n \\
    b_1   & 0         &\cdots & 0 \\
    \vdots & \vdots    &\ddots & \vdots\\
    b_n   & 0        & \cdots     & 0 \\
  \end{pmatrix}.
$$
Then a direct computation of the character polynomial of $U+tV$ shows that  $U+tV$ has the same eigenvalue multiplicities
as $U$ only when $t=0$, which is a contradiction. This proves the \textbf{Assertion} for $A_n$,
$n>1$.

{\bf The case $\mathfrak{g}=D_n$ with $n\geq 4$}. The Cartan subalgebra $\mathfrak{t}$
with the bi-invariant metric can be
 realized as the standard Euclidean space $\mathbb{R}^n$. Let $\{e_1,\ldots,e_n\}$ be the standard orthonormal basis. The root system
will be identified with its dual set in $\mathfrak{t}=\mathbb{R}^n$, i.e.,
\begin{equation}\label{rootsystemDn}
\{\pm e_i\pm e_j, \forall 1\leq i<j\leq n\}.
\end{equation}
Suppose $U=\mathop{\sum}\limits_{i=1}^n a_i e_i$ and $X=\mathop{\sum}\limits_{i=1}^{n} b_i e_i$.

If there is a pair $\{i,j\}$, such that
$|a_i|\neq |a_j|$ and $|b_i|\neq |b_j|$, then the four roots $\pm e_i\pm e_j$ are not
orthogonal to either $U$ or $X$,  hence the \textbf{Assertion} holds.
Now we suppose  conversely that the \textbf{Assertion} is not true. Then the above argument shows that the following two assertions hold:
\begin{enumerate}
\item\quad If  there exists $i<j$ such that $|a_i|\ne |a_j|$, then for any $k, l$, we have  $|b_k|=|b_l|$;
 \item\quad  If  there exists $i<j$ such that $|b_i|\ne |b_j|$, then for any $k, l$, we have  $|a_k|=|a_l|$.
\end{enumerate}
 By exchanging $U$ and $X$, multiplying $U$ and $X$ by suitable nonzero scalars, or using a Weyl group action to change $U$ and $X$ simultaneously (i.e., reordering the entries and
changing the signs of even entries), we can reduce the discussion to the two cases below. We will see that, in either case, one can find four roots which are not orthogonal to either $U$ or $X$. If the Weyl group action is used, then we use its inverse action to pull the roots back, and we can find the four roots indicated by the \textbf{Assertion}
for the original $U$ and $X$.

(1)\quad  $a_i=1$, $\forall i=1,\ldots,n$.
If $b_1\neq 0$ and $b_2=b_3=0$, then we choose the roots $\pm(e_1+e_2)$ and $\pm(e_1+e_3)$.
If the first two $b_i$s are nonzero
and $b_3=0$, then we choose $\pm(e_1+e_3)$ and $\pm(e_2+e_3)$.
If the first three $b_i$ are all nonzero and have the same sign, then we
choose the roots $\pm(e_1+e_2)$ and $\pm(e_1+e_3)$.
If the first four $b_i$ are all nonzero and $\mathrm{sign}(b_1)=\mathrm{sign}(b_2)\neq\mathrm{sign}(b_3)
=\mathrm{sign}(b_4)$, then we choose
the roots $\pm(e_1+e_2)$ and $\pm(e_3+e_4)$.

(2)\quad  $a_i=-1$ and  $a_i=1$, for any $i\ne 1$. We can further change the sign of $a_1$ and
$b_1$ simultaneously. Changing the sign of only one entry is not a Weyl group action, but
it preserves the root system. Then the discussion is reduced to (1).

To summarize, for all cases of $U$ and $X$, we can deduce a contradiction if we assume
that the \textbf{Assertion} is not true. This proves the \textbf{Assertion} for $D_n$, $n\geq 4$.

{\bf The case $\mathfrak{g}=B_n$ with $n>2$}. With the Cartan subalgebra identified with
the standard Euclidian space $\mathbb{R}^n$, the root system of $B_n$ can be identified with
the set
\begin{equation}\label{rootsystemBn}
\{\pm e_i\pm e_j, \pm e_i, \forall i\neq j\}.
\end{equation}
When $n\geq 4$, the root system of $B_n$ in (\ref{rootsystemBn})
contains that of $D_n$ in (\ref{rootsystemDn}), and the \textbf{Assertion} follows from that of the case of $D_n$.

Now consider $B_3$. Assume conversely that there does not exist four roots which are not orthogonal to either $U$ or $X$. Let $U=a_1 e_1+a_2 e_2 + a_3 e_3$ and $X=b_1 e_1 + b_2 e_2
+b_3 e_3$. Using a similar  argument as above, we can show that
either  $|a_1|=|a_2|=|a_3|$ or  $|b_1|=|b_2|=|b_3|$. Then we can similarly exchange
$U$ and $X$, or change $U$ and $X$ by nonzero scalar multiplications, or use Weyl group actions
to change $U$ and $X$ simultaneously (that is, reorder the entries and change the signs arbitrarily), to reduce the discussion to the following two cases.
\begin{description}
\item{(i)}\quad $a_1=a_2=a_3=1$, $b_1\neq 0$ and $b_2=b_3=0$. In this case,  we choose $\pm e_1$ and $\pm(e_1+e_2)$.
\item{(ii)}\quad $a_1=a_2=a_3=1$,  and the first two $b_i$ are nonzero. In this case,  we choose $\pm e_1$ and $\pm e_2$.
\end{description}
Then in both  cases the selected four roots are not orthogonal to either $U$ or $X$. This leads to a contradiction, proving  the \textbf{Assertion} for $B_n$, $n>2$.

{\bf The case $\mathfrak{g}=C_n$ with  $n>2$}. The argument is exactly the same as for the previous case.

{\bf The case $\mathfrak{g}=B_2=C_2$ or $G_2$}. The number of roots orthogonal to $U$ or $X$ is at most four. Therefore  there are at least $4$ roots which are not orthogonal to either $U$ or $X$.

{\bf The case $\mathfrak{g}=F_4$}.  With the Cartan subalgebra identified with
the standard Euclidian space $\mathbb{R}^4$, the root system of $F_4$ can be identified with a
set consisting of $48$ vectors, namely,   all permutations of $(\pm 1,\pm 1,0,0)$, all permutations
of $(\pm 1,0,0,0)$, and $(\pm \frac{1}{2},\pm \frac{1}{2},\pm \frac{1}{2},\pm \frac{1}{2})$. Since it contains the root system
of $B_4$ in (\ref{rootsystemBn}),
the \textbf{Assertion} follows.

{\bf The case  $\mathfrak{g}=E_6$}. When the Cartan subalgebra $\mathfrak{t}$ is modeled as the Euclid space
$\mathbb{R}^6$ with the standard inner product, the root system consists of the following two sets:
\begin{description}
\item{(1)}\quad The $40$ roots in the root system of $D_5$, i.e.,
 $(c_1,c_2,c_3,c_4,c_5,0)$, where two of the $c_i$, $i=1,\cdots,5$, are $\pm 1$ and all the others  are $0$;
 \item{(2)}\quad  The $32$ roots of the form $(\pm \frac{1}{2},\ldots,\pm \frac{1}{2},\pm \frac{\sqrt{3}}{2})$, where the  total number of the plus signs is odd.
\end{description}
Assume conversely that there does not exist four roots which are not orthogonal to either $U$ or $X$. Let $U=(U',a_6)=(a_1,\ldots,a_6)$ and $X=(X',b_6)=(b_1,\ldots,b_6)$.
 Note that  the root system of $E_6$ contains that of $D_5$. By the argument in the case of
$D_n$, we have either $U'=0$ or $X'=0$.
Using a suitable scalar change or an exchange between $U$ and $X$ if necessary, we can assume that
$U=(0,\ldots,0,1)$. If $X'=0$, then any root of the form $(\pm \frac{1}{2},\ldots,\pm \frac{1}{2},\pm \frac{\sqrt{3}}{2})$ is not orthogonal to either $U$ or $X$. Without losing generality,
we can assume that $b_6\leq 0$ and $X'\neq 0$. Furthermore,   we can use the Weyl group
 of $D_5$ (viewed as a subgroup of the Weyl group of $E_6$) to change $X'$ so that $b_1\geq \cdots\geq b_5$,
$|b_1|\geq\cdots\geq |b_5|$, $b_1>0$ and $b_4\geq 0$, without changing $U$. Then the
four roots
$\pm(\frac{1}{2},\ldots,\frac{1}{2},-\frac{\sqrt{3}}{2})$ and $\pm(\frac{1}{2},\frac{1}{2},\frac{1}{2},-\frac{1}{2},-\frac{1}{2},-\frac{\sqrt{3}}{2})$ are not
 orthogonal to either $U$ or $X$, which is a contradiction. This proves
the \textbf{Assertion} for $E_6$.

{\bf The case $\mathfrak{g}=E_7$}. When $\mathfrak{t}$ is modeled as the Euclid space $\mathbb{R}^7$ with the standard inner product, the root system consists of the roots of the following three types:
\begin{description}
\item{(a)}\quad The roots $(c_1,c_2,c_3,c_4,c_5, c_6,0)$ in the root system $D_6$ such that two of $c_i$, $i=1,\cdots,6$ are equal to $\pm1$ and all the others are $0$;
\item{(b)}\quad  The elements of the form
$(\pm \frac{1}{2},\pm \frac{1}{2},\ldots,\pm \frac{1}{2},\pm \frac{1}{\sqrt{2}})$,
where  the total  number of the $+\frac{1}{2}$ is even;
\item{(c)}\quad  The two elements $(0,\ldots,0,\pm\sqrt{2})$.
\end{description}

Let $U=(U',a_7)=(a_1,\ldots,a_7)$ and $X=(X',b_7)=(b_1,\ldots,b_7)$. Suppose conversely that there does not exists four roots which are not orthogonal to either $U$ or $X$.  Note that the root system of $E_7$ contains that of $D_6$, the argument for the case of $D_n$, $n\geq 4$ then indicates that either $U'=0$ or $X'=0$.
If $U'=X'=0$, then any root  of the type (b) or (c) is not orthogonal to either $U$
or $X$, which is a  contradiction.
 With a possible exchange between $U$ and $X$, and nonzero scalar changes, we can assume that $U=(0,\ldots,0,1)$, $X'\neq 0$ and $b_7\leq 0$.
Using the Weyl group of $D_6$ (viewed as a subgroup of the Weyl group of $E_7$) to change $X'$ while keeping $b_6$ and $U$  unchanged,  we can assume that $b_1\geq b_2\cdots \geq b_6$, $|b_1|\geq\cdots\geq |b_6|$, $b_1>0$ and $b_5\geq 0$.
Then the roots
$\pm(\frac{1}{2},\ldots,\frac{1}{2},-\frac{1}{\sqrt{2}})$ and $\pm(\frac{1}{2},\ldots,\frac{1}{2},-\frac{1}{2},-\frac{1}{2},-\frac{1}{\sqrt{2}})$ are not orthogonal to either $U$ or $X$, which is a contradiction. This proves the \textbf{Assertion} for $E_7$.

{\bf The case $\mathfrak{g}=E_8$}. With the Cartan subalgebra $\mathfrak{t}$ identified with
the standard $\mathbb{R}^8$, the root system of $E_8$ consists of the vectors of the following types:
\begin{description}
\item{(a)}\quad Elements of  the
root system of $D_8$, i.e., all the permutations of $(\pm 1,\pm 1,0,\ldots,0)$;
\item{(b)}\quad  Elements of the form
$(\pm \frac{1}{2},\pm \frac{1}{2},\ldots,\pm \frac{1}{2})$, where the total number of the minus signs is even.
\end{description}
It contains the root system of $D_8$ in (\ref{rootsystemDn}). Thus the \textbf{Assertion} in this case follows from the argument in the
case of $D_n$, $n\geq 4$.

Up to now we have completed the proof of the \textbf{Assertion} for all the cases. This completes the proof of
Lemma \ref{key lemma}, concluding the proof of all the results in this paper.

\medskip
\textbf{Acknowledgements:}\quad This work was finished during the first author's visit to the Chern Institute of Mathematics. He is grateful to the faculty members of the institute for their hospitality.

\end{document}